\documentclass{article}
\usepackage{mysty}
\usepackage[square, numbers]{natbib}
\numberwithin{equation}{section}
\renewcommand{\S}{\mathbb{S}}
\renewcommand{\u}{\mathbf{u}}
\renewcommand{\v}{\mathbf{v}}
\renewcommand{\a}{\alpha}
\renewcommand{\b}{\beta}

\newcommand{\T}{\top}
\newcommand{\e}{e}
\newcommand{\eps}{\varepsilon}
\newcommand{\1}{\mathbf{1}}
\newcommand{\SDP}{\mathcal{Q}}

\begin{document}
\title{Four-point semidefinite bound for equiangular lines}
\author{Wei-Jiun Kao\thanks{National Center for Theoretical Sciences, Mathematics Division, Taipei, Taiwan. {\tt wjkao@ncts.ntu.edu.tw}} \qquad Wei-Hsuan Yu\thanks{Mathematics Department, National Central University, Taoyuan, Taiwan. {\tt u690604@gmail.com}}}
\date{}

\maketitle
\begin{abstract}
A set of lines in $\R^d$ passing through the origin is called equiangular if any two lines in the set form the same angle. We proved an alternative version of the three-point semidefinite constraints developed by~\citeauthor{bachoc2008new}, and the multi-point semidefinite constraints developed by~\citeauthor{musin2013multivariate} for spherical codes. The alternative semidefinite constraints are simpler when the concerned object is a spherical $s$-distance set. Using the alternative four-point semidefinite constraints, we found the four-point semidefinite bound for equiangular lines. This result improves the upper bounds for infinitely many dimensions $d$ with prescribed angles. As a corollary of the bound, we proved the uniqueness of the maximum construction of equiangular lines in $\R^d$ for $7 \leq d \leq 14$ with inner product $\a = 1/3$, and for $23 \leq d \leq 64$ with $\a = 1/5$.
\end{abstract}

\section{Introduction}
A set of lines in $\R^d$ passing through the origin is called {\bf equiangular} if any two lines in the set form the same angle. We denote the maximum cardinality of sets of equiangular lines in $\R^d$ by $N(d)$. Searching the values of $N(d)$ is one of the classical problems in discrete geometry. The history can be traced back to~\citeyear{haantjes1948equilateral} by the work of \citet{haantjes1948equilateral}, who settled the problem for $d=3$ and $d=4$. To the best of our knowledge, the results for the bounds for $N(d)$ are summarized in Table~\ref{tb:smallnd}; see Appendix~\ref{sec: review} for the details. One can also refer to Sequence A002853 in The On-Line Encyclopedia of Integer Sequences~\cite{oeis} for the latest results.

\begin{table}[h]
	\centering
    \label{tb:smallnd}
    \begin{tabular}{c|cccccccc}
   		$d$    & 2 & 3--4 & 5  & 6  & 7--14     & 15 & 16     & 17 \\ \hline
        $N(d)$ & 3 & 6    & 10 & 16 & 28 & 36 & 40 & 48 \\
        \hline\hline
        $d$    & 18     & 19     & 20     & 21  & 22  & 23--41 & 42       & 43 \\ \hline
        $N(d)$ & 57--60 & 72--74 & 90--94 & 126 & 176 & 276    & 276--288 & 344
    \end{tabular}
    \caption{The maximum cardinality of equiangular lines for small dimensions $d$}
\end{table}

For a set of equiangular lines in $\R^d$, let $X \subseteq \S^{d-1}$ be a set of unit vectors constructed by choosing a unit vector along every line. Then the set $X$ satisfies
\[
    c \cdot c' = \pm \a \text{ for all } c, c' \in X \text{ and } c \neq c'
\]
for some real number $\a \in [0, 1)$. The set $X$ is called a {\bf spherical projection} of the original set of equiangular lines. A spherical projection is a spherical $2$-distance set, i.e., a set of unit vectors with the {\bf inner product set}
\[
    A(X) \coloneqq \{c \cdot c': c, c' \in X, c \neq c'\}
\]
containing two different elements. A set $X \subseteq \S^{d-1}$ of unit vectors is a {\bf spherical $s$-distance set} if $|A(X)| = s$.

Let $N_{\a}(d)$ be the maximum cardinality of sets of equiangular lines in $\R^d$ with prescribed inner products $\pm\a$. Neumann~\cite{lemmens1973equiangular} proved that, for a set of equiangular lines in $\R^d$ with cardinality greater than $2d$, the reciprocal of the inner product $\a$ must be an odd integer. Therefore, the values of $N_{1/a}(d)$ for odd integers $a$ are essential for determining $N(d)$. Below are some partial results on $N_{1/a}(d)$; see Appendix~\ref{sec: review2} for the details.

\citet{lemmens1973equiangular} completely determined $N_{1/3}(d)$ for all dimensions $d$ as in Table~\ref{tb: N13d}. Note that the maximum cardinalities remain the same from $d = 7$ to $d = 15$, and become a linear function in $d$ for $d \geq 15$. There is a similar phenomenon when $\a = 1/5$: \citet{cao2022lemmens} proved that $N_{1/5}(d) = 276$ for $23 \leq d \leq 185$, and $N_{1/5}(d) = \lfloor\frac{3}{2}(d-1)\rfloor$ for $d \geq 185$.

\begin{table}[h]
    \centering
    \begin{tabular}{c|cccccc}
        $d$ & 3 & 4 & 5 & 6 & 7-15 & 15- \\ \hline
        $N_{1/3}(d)$ & 4 & 6 & 10 & 16 & 28 & $2d-2$
    \end{tabular}
    \caption{The maximum cardinality of equiangular lines for $\a = 1/3$}
    \label{tb: N13d}
\end{table}

In general, the asymptotic behaviors of $N_{\a}(d)$ are also determined.
\begin{thm}[\citet{jiang2021equiangular}]
    Let $a \geq 3$ be an odd integer. Then $N_{1/a}(d) = \lfloor\frac{a+1}{a-1}(d-1)\rfloor$ for all sufficiently large dimensions $d$.
\end{thm}

As for the constant upper bound counterpart, there is a theorem proved by applying the three-point semidefinite programming method to spherical projections of sets of equiangular lines.
\begin{thm}[\citet{yu2017new}]\label{thm: 3sdp_main_mirror}
Let $a \geq 3$ be an odd integer. Suppose $d \leq  D_3(a) = 3a^2-16$. Then
\[
    N_{1/a}(d) \leq \frac{1}{2}(a^2-1)(a^2-2).
\]
\end{thm}

The three-point semidefinite programming method is initially developed by~\citet{bachoc2008new} to bound the kissing numbers in $\R^d$. The method consists of some constraints on the three-point distance distribution, i.e., the number of triples in a spherical set $X$ with three specified values for pairwise inner products. Theorem~\ref{thm: 3sdp_main_mirror} is called the {\bf three-point semidefinite bound} for equiangular lines.

The aim of this paper is to generalize Theorem~\ref{thm: 3sdp_main_mirror}. We consider the multi-point distance distributions, which must satisfy some constraints developed by~\citet{musin2013multivariate}. We proved an alternative version of the semidefinite constraints. The {\bf alternative semidefinite constraints} are simpler when the concerned object $X$ is a spherical $s$-distance set. When $X$ is a spherical projection of a set of equiangular lines, the constraints can be simplified further by considering the mutual relations of different spherical projections. We call this simplification by {\bf switching reduction}. Using the alternative four-point semidefinite constraints with switching reduction, we proved the following {\bf four-point semidefinite bound} for equiangular lines.
\begin{thm}\label{thm: 4sdp_main_mirror}
Let $a \geq 3$ be an odd integer. Suppose $d \leq D_4(a) = 3a^2 + (12/\sqrt{5})a - 948/25 + o(1)$. Then
\[
    N_{1/a}(d) \leq \frac{1}{2}(a^2-1)(a^2-2).
\]
\end{thm}

The four-point semidefinite bound is a generalization of the three-point semidefinite bound, in the sense that Theorem~\ref{thm: 4sdp_main_mirror} gives upper bounds for $N_{1/a}(d)$ for infinitely many pairs $(a, d)$ more than Theorem~\ref{thm: 3sdp_main_mirror}. To be specific, the largest applicable dimensions $d$ are improved from $D_3(a) = 3a^2 + O(1)$ to $D_4(a) = 3a^2+(12/\sqrt{5})a+O(1)$. The improvements for small $a$ are shown in Table~\ref{tb: applicable_dim}.

\begin{table}[h]
    \centering
    \begin{tabular}{c|ccccc}
        $a$ & 3 & 5 & 7 & 9 & 11 \\
        \hline
        $D_3(a)$ & 11 & 59 & 131 & 227 & 347 \\
        $D_4(a)$ & 14.42 & 64.56 & 144.52 & 250.41 & 380.96 \\
    \end{tabular}
    \caption{$D_3(a)$ and $D_4(a)$ for small $a$}
    \label{tb: applicable_dim}
\end{table}

The four-point semidefinite bound agrees with the work by~\citet{de2021k}. They gave the upper bounds for $N_{1/a}(d)$ for $a = 5, 7, 9, 11$ by the four-point semidefinite programming method. By this method, they showed that $N_{1/5}(65) \leq 276$, $N_{1/7}(145) \leq 1128$, $N_{1/9}(251) \leq 3160$ and $N_{11}(381) \leq 7140$ numerically. Theorem~\ref{thm: 4sdp_main_mirror} serves as a rigorous proof of their results and is applicable for infinitely many $a$.

If a set of equiangular lines is a maximum construction attaining the equality in Theorem~\ref{thm: 4sdp_main_mirror}, we can derive more information about it.
\begin{thm}\label{thm: maximum_mirror}
    Let $a \geq 3$ be an odd integer and $d < D_4(a)$. Then a set of equiangular lines in $\R^d$ with cardinality $(a^2-1)(a^2-2)/2$ and inner product $1/a$ must lie in an $(a^2-2)$-dimensional subspace.
\end{thm}
By Theorem~\ref{thm: maximum_mirror} along with the uniqueness of the strongly regular graphs with some specific parameters, we proved the uniqueness of the maximum construction of equiangular lines in $\R^d$ for $7 \leq d \leq 14$ with $\a = 1/3$, and for $23 \leq d \leq 64$ with $\a = 1/5$.

The article is organized as follows: In Section~\ref{sec: sdp}, we review the three-point semidefinite constraints developed by~\citet{bachoc2008new}, and the multi-point semidefinite constraints developed by~\citet{musin2013multivariate}. We also prove an alternative version of these constraints in Section~\ref{sec: sdp}. In Section~\ref{sec: switching} we develop the switching reduction on the regime that $X$ is a spherical projection of a set of equiangular lines. In Section~\ref{sec: 4sdp} we give the proof of the main results, namely, Theorem~\ref{thm: 4sdp_main_mirror} and Theorem~\ref{thm: maximum_mirror}. Section~\ref{sec: questions} contains some further questions.

In the following text, $X$ is a spherical $s$-distance set in $\S^{d-1}$ when not specified, and the cardinality is $N = |X|$. The inner product set of $X$ is defined by
\[
    A(X) \coloneqq \{ c \cdot c': c, c' \in X, c \neq c'\}.
\]
So $|A(X)| = s$. We also write $A'(X) = A(X) \cup \{1\}$. When it is specified that $X$ is a spherical projection of a set of equiangular lines in $\R^d$, we fix $A(X) = \{\pm\a\}$ and $a = 1/\a$.
\section{Semidefinite constraints}\label{sec: sdp}

The semidefinite programming method is developed by~\citet{bachoc2008new}. This method is a generalization of the linear programming method developed by~\citet{delsarte1977spherical}.

The two methods are initially developed to bound the kissing numbers in $\R^d$, that is, the maximum cardinality of $X \subseteq \S^{d-1}$ such that $A(X) \subseteq [-1, 1/2]$. In~\citeyear{odlyzko1979new},~\citet{odlyzko1979new} and \citet{levenshtein1979bounds} independently proved that the kissing numbers in $\R^8$ and $\R^{24}$ are $240$ and $196560$ using the linear programming method. \citet{musin2008kissing} proved that the kissing number in $\R^4$ is $24$ by a modification of the linear programming method. Meanwhile,~\citet{bachoc2008new} gave the same result by the semidefinite programming method.

The linear programming method and the semidefinite programming method involve the following parts:
\begin{enumerate}[label = \arabic*.]
    \item Consider the {\bf distance distribution} of a spherical set $X$.
    \item The distance distribution is related to the cardinality $N = |X|$.
    \item Meanwhile, the distance distribution satisfies some specific inequalities or semidefinite constraints.
    \item The upper bound for the cardinality $N$ is given by considering a linear program or a semidefinite program with $N$ being the objective function and the distance distribution being the variables. 
\end{enumerate}

This framework can also be applied to spherical $s$-distance sets, as well as spherical projections of sets of equiangular lines. See~\cite{barg2013new, barg2013new_eq, glazyrin2018upper, okuda2016new, yu2017new} for some upper bounds on spherical $2$-distance sets and equiangular lines proved by using the semidefinite programming method.

\citeauthor{bachoc2008new}'s semidefinite programming method involves the three-point distance distribution satisfying the so-called three-point semidefinite constraints. \citet{musin2013multivariate} generalized the method and developed the multi-point semidefinite constraints of $(m+2)$-point distance distributions for $m \geq 1$. As an application,~\citet{de2021k} gave some numerical upper bounds on the cardinality of sets of equiangular lines by $(m+2)$-point semidefinite programming method for $1 \leq m \leq 4$.

In this section, we review the linear inequalities and the semidefinite constraints and prove an alternative version of the semidefinite constraints. The alternative semidefinite constraints (Theorem~\ref{thm: 3sdp_alternative} and Theorem~\ref{thm: msdp_alternative}) are simpler than the original ones (Theorem~\ref{thm: 3sdp_original} and Theorem~\ref{thm: msdp_original}) when the concerned object $X$ is a spherical $s$-distance set.

\subsection{Two-point linear inequality}

{\bf Part 1.} Consider the {\bf two-point distance distribution}
\[
    x(t) \coloneqq \#\{(c, c') \in X^2: c \cdot c' = t\}.
\]
The value $x(t)$ counts the number of pairs in $X$ with inner product $t$. Clearly $x(1) = N$, and $x(t) > 0$ only when $t = 1$ or $t \in A(X)$.

{\bf Part 2.} $\sum_{t \in A(X)} x(t) = N(N-1)$.

{\bf Part 3.} The inequalities for the distance distribution $x(t)$ arise from the property of a series of polynomials called {\bf Gegenbauer polynomials}. The Gegenbauer polynomials $\{P^d_k(t): k = 0, 1, 2, \dots\}$ are defined by
\begin{align*}
    P^d_0(t) &= 1, \\
    P^d_1(t) &= t, \\
    P^d_k(t) &= \frac{1}{k+d-3}\left((2k+d-4)P^d_{k-1}(t) - (k-1)P^d_{k-2}(t)\right),\ k \geq 2.
\end{align*}

For $c, c' \in \S^{d-1}$, $P^d_k(c \cdot c')$ can be realized as a specific inner product (\citet{delsarte1977spherical}). Therefore
\begin{equation}\label{eq: add_formula}
    \sum_{c, c' \in X} P^d_k(c \cdot c') \geq 0 \text{ for } k = 0, 1, 2, \dots.
\end{equation}
By rewriting the inequality~\eqref{eq: add_formula} in terms of the two-point distance distribution, we have the following theorem.

\begin{thm}\label{thm: 2lp}
    For a spherical set $X$ in $\S^{d-1}$ with inner product set $A$, $\sum_{t \in A} x(t) = N(N-1)$ and
    \[
        N + \sum_{t \in A} x(t)P^d_k(t) \geq 0 \text{ for } k \geq 1.
    \]
\end{thm}

{\bf Part 4.} One can establish a linear program using Theorem~\ref{thm: 2lp} to bound the cardinality $N$.

\subsection{Three-point semidefinite constraint}

{\bf Part 1.} Consider the {\bf three-point distance distribution}
\[
    y(u, v, t) \coloneqq \#\{(b, c, c') \in X^3: b \cdot c = u, b \cdot c' = v, c \cdot c' = t\}.
\]
The value $y(u, v, t)$ counts the number of triples in $X$ with pairwise inner products $u$, $v$ and $t$.

{\bf Part 2.} The distance distributions are related to each other and the cardinality $N$ by
\begin{itemize}
    \item $y(1, 1, 1) = x(1) = N$,
    \item $y(1, t, t) = y(t, 1, t) = y(t, t, 1) = x(t)$,
    \item $y(1, u, v) = y(u, 1, v) = y(u, v, 1) = 0$ for $u \neq v$, and
    \item $\sum_{u, v, t \in A(X)} y(u, v, t) = N(N-1)(N-2)$.
\end{itemize}

{\bf Part 3.} \citet{bachoc2008new} proved that, there is a specific series of matrices $\{Y^d_{k, n}: k \geq 0, n \geq 0\}$ such that
\begin{equation}\label{eq: 3sdp_doublesum}
    \sum_{c, c' \in X} Y^d_{k, n}(b \cdot c, b \cdot c', c \cdot c') \succeq 0
\end{equation}
for any fixed point $b \in \S^{d-1}$. The matrices $Y^d_{k, n}$ are defined by
\begin{align*}
    Y^d_{k, n}(u, v, t) &\coloneqq Q^d_k(u, v, t) \begin{pmatrix}
    1 \\ u \\ \vdots \\ u^n
    \end{pmatrix}
    \begin{pmatrix}
    1 & v & \cdots & v^n
    \end{pmatrix}, \\
    Q^d_k(u, v, t) &\coloneqq (1-u^2)^{k/2}(1-v^2)^{k/2}P^{d-1}_k\left(\frac{t-uv}{\sqrt{(1-u^2)(1-v^2)}}\right).
\end{align*}
The order of $Y^d_{k, n}(u, v, t)$ is $n+1$, and in which every entry is a polynomial of $u, v, t$ in degree $k$.

By summing $b$ in~\eqref{eq: 3sdp_doublesum} over elements of $X$, we have the following constraints about $y(u, v, t)$.

\begin{thm}[\citet{bachoc2008new}]\label{thm: 3sdp_original}
For a spherical set $X$ in $\S^{d-1}$ with inner product set $A$,
\[
    \sum_{u, v, t \in A}y(u, v, t) = (N-2)\sum_{t \in A}x(t)
\]
and
\begin{align*}
    & NY^d_{k, n}(1, 1, 1) & \\
    +& \sum_{t \in A}x(t)\Big(Y^d_{k, n}(1, t, t) + Y^d_{k, n}(t, 1, t) + Y^d_{k, n}(t, t, 1)\Big) & \\
    +& \sum_{u, v, t \in A} y(u, v, t)Y^d_{k, n}(u, v, t) & \succeq 0 \text{ for } k, n \geq 0.
\end{align*}
\end{thm}
{\bf Part 4.} One can establish a semidefinite program using Theorem~\ref{thm: 3sdp_original} to bound the cardinality $N$. In \citet{bachoc2008new}, the variables of the semidefinite program are $x(t)/N$ and $y(u, v, t)/N$, and the objective function is $1 + \sum_t x(t)/N = N$.

Below shows an alternative version of Theorem~\ref{thm: 3sdp_original}.

{\bf Part 3 (alternative version).} Let $A' = A \cup \{1\}$. The semidefinite constraint in Theorem~\ref{thm: 3sdp_original} is equivalent to

\[
    \sum_{u, v, t \in A'}y(u, v, t)Q^d_k(u, v, t) \begin{pmatrix}1 \\ u \\ \vdots \\ u^n\end{pmatrix}
    \begin{pmatrix}1 & v & \cdots & v^n\end{pmatrix} \succeq 0.
\]
The positive semidefiniteness means that
\[
    \textbf{a}^{\T}
    \left(\sum_{u, v, t \in A'}y(u, v, t)Q^d_k(u, v, t) \begin{pmatrix}1 \\ u \\ \vdots \\ u^n\end{pmatrix}
    \begin{pmatrix}1 & v & \cdots & v^n\end{pmatrix}\right) \textbf{a} \geq 0
\]
for all $\textbf{a} \in \R^{n+1}$, or
\begin{equation}\label{eq: 3sdp_altform_1}
    \sum_{u, v, t \in A'} y(u, v, t)Q^d_k(u, v, t)f(u)f(v) \geq 0 
\end{equation}
for all polynomials $f(x)$ with degree at most $n$. Since $n$ can be arbitrary, the inequality~\eqref{eq: 3sdp_altform_1} holds for any polynomial $f(x)$.

Let $\textstyle\{\e_u: u \in A'\}$ be the standard basis of $\R^{s+1}$. We rewrite \eqref{eq: 3sdp_altform_1} as
\[
    \sum_{u, v, t \in A'} y(u, v, t)Q^d_k(u, v, t) (\mathbf{f}^{\T}\e_u)(\e_v^{\T}\mathbf{f}) \geq 0
\]
where $\textstyle \mathbf{f} = \sum_{u \in A'} f(u)\e_u$ for some polynomial $f(x)$. In fact, $\mathbf{f}$ can be taken as arbitrary vector in $\R^{s+1}$; that is, given any vector $\textstyle \mathbf{f} = \sum_{u \in A'} f_u\e_u$, we can always find a polynomial $f(x)$ such that
\[
    f(u) = f_u,\ u \in A'.
\]
Therefore, the matrix
\begin{equation}\label{eq: 3sdp_altform_2}
    \sum_{u, v, t \in A'} y(u, v, t)Q^d_k(u, v, t)\e_u\e_v^{\T}
\end{equation}
is positive semidefinite. We denote the matrix~\eqref{eq: 3sdp_altform_2} by $\SDP^d_k(X)$, and call it by the {\bf alternative three-point semidefinite constraint}.

\begin{thm}[Alternative three-point semidefinite constraint]\label{thm: 3sdp_alternative}
For a spherical $s$-distance set $X$ in $\S^{d-1}$ with inner product set $A$,
\[
    \sum_{u, v, t \in A}y(u, v, t) = (N-2)\sum_{t \in A}x(t).
\]
Let $\{e_u: u \in A \cup \{1\}\}$ be the standard basis of $\R^{s+1}$. Then
\begin{align*}
    \SDP^d_k(X) &= NQ^d_k(1, 1, 1) \e_1\e_1^{\T} & \\
    +& \sum_{t \in A} x(t)\Big(Q^d_k(1, t, t)\e_1\e_t^{\T} + Q^d_k(t, 1, t)\e_t\e_1^{\T} + Q^d_k(t, t, 1)\e_t\e_t^{\T}\Big) & \\
    +& \sum_{u, v, t \in A} y(u, v, t)Q^d_k(u, v, t) \e_u\e_v^{\T} & \succeq 0
\end{align*}
for any $k \geq 0$.
\end{thm}

\begin{rmk}\label{rmk: alternative_simpler}
When the concerned object $X$ has a finite number of distances, the alternative semidefinite constraints $\SDP^d_k(X)$ are simpler than the original ones in Theorem~\ref{thm: 3sdp_original}. For a spherical $s$-distance set $X$, the original constraints consist of sums of $s^3 + 3s + 1$ matrices with order $(n+1)$; in the alternative constraints $\SDP^d_k(X)$, there are only $s^3 + 3s + 1$ terms in a $(s+1) \times (s+1)$ matrix.
\end{rmk}

\begin{rmk}
One might think that the original semidefinite constraints are better with larger matrix size $(n+1)$. However, an argument similar to the proof of Theorem~\ref{thm: 3sdp_alternative} indicates that the restricting ability on $y(u, v, t)$ of the original semidefinite constraints is the same as the alternative constraints whenever $n \geq s$. Therefore, it suffices to take $n=s$ when applying the original semidefinite constraints.
\end{rmk}

\begin{eg}\label{eg: 3sdp_alternative}
Consider a $2$-distance set $X$ with $A(X) = \{\a, \b\}$. The alternative three-point semidefinite constraint is
\[
    \SDP^d_k(X) = \begin{pmatrix*}[r]
    NQ^d_k(1,1,1)
    & x(\a)Q^d_k(1,\a,\a)
    & x(\b)Q^d_k(1,\b,\b) \\[1em]
    x(\a)Q^d_k(\a,1,\a)
    & \begin{matrix*}[r] \phantom{+} x(\a)Q^d_k(\a,\a,1) \\ + y(\a,\a,\a)Q^d_k(\a,\a,\a) \\ + y(\a,\a,\b)Q^d_k(\a,\a,\b)\end{matrix*}
    & \begin{matrix*}[l] \phantom{+} y(\a,\b,\a)Q^d_k(\a,\b,\a) \\ + y(\a,\b,\b)Q^d_k(\a,\b,\b) \end{matrix*} \\[1em]
    x(\b)Q^d_k(\b,1,\b)
    & \begin{matrix*}[l] \phantom{+} y(\b,\a,\a)Q^d_k(\b,\a,\a) \\ + y(\b,\a,\b)Q^d_k(\b,\a,\b) \end{matrix*}
    & \begin{matrix*}[r] \phantom{+} x(\b)Q^d_k(\b,\b,1) \\ + y(\b,\b,\a)Q^d_k(\b,\b,\a) \\ + y(\b,\b,\b)Q^d_k(\b,\b,\b) \end{matrix*}
    \end{pmatrix*}\succeq 0.
\]
\end{eg}

%%%%%%%%%%%%%%%%%%%%%%%%%%%%%

\subsection{Multi-point semidefinite constraint}

The multi-point semidefinite constraints are generalizations of the three-point semidefinite constraints. Fix $m$ to be a positive integer with $1 \leq m \leq d-2$. The multi-point constraints become the three-point constraints when $m=1$.

{\bf Part 1.} We use the notation of {\bf Gram matrices} to define the {\bf multi-point distance distribution}.

\begin{defn} Suppose $B = \{b_1, \dots, b_m\} \subseteq X$ is a set of unit vectors, and $c, c' \in \S^{d-1}$.
\begin{enumerate}[label=(\alph*)]
    \item The Gram matrix of $B$ is an $m \times m$ matrix
    \[
        G = G(B) \coloneqq \begin{pmatrix}b_p \cdot b_q\end{pmatrix}_{1 \leq p, q \leq m}.
    \]
    The Gram matrix $G$ is symmetric and positive semidefinite. If $G$ is positive definite, i.e., $B$ is linearly independent, we say $G$ is a {\bf proper Gram matrix}.
    \item For $c \in \S^{d-1}$, denote the vector $B^{\T} \cdot c \coloneqq (b_1 \cdot c, \dots, b_m \cdot c)^{\T}$ in $\R^m$.
    \item Suppose $\u = B^{\T} \cdot c \in \R^m$, then the Gram matrix of $B \cup \{c\}$ is
    \[
        G(B, c) = \begin{pmatrix}
            G(B) & B^{\T} \cdot c \\
            c^{\T} \cdot B & c \cdot c
        \end{pmatrix} = \begin{pmatrix}
            G & u \\
            u^{\T} & 1
        \end{pmatrix}.
    \]
    We denote the $(m+1) \times (m+1)$ Gram matrix by the pair $(G; \u)$.
    \item Suppose $\u = B^{\T} \cdot c$ and $\v = B^{\T} \cdot c'$, then the Gram matrix of $B \cup \{c, c'\}$ is
    \[
        G(B, c, c') = \begin{pmatrix}
            G(B) & B^{\T} \cdot c & B^{\T} \cdot c' \\
            c^{\T} \cdot B & c \cdot c & c \cdot c' \\
            c'^{\T} \cdot B & c' \cdot c & c' \cdot c'
        \end{pmatrix} = \begin{pmatrix}
            G & \u & \v \\
            \u^{\T} & 1 & t \\
            \v^{\T} & t & 1
        \end{pmatrix}.
    \]
    We denote the $(m+2) \times (m+2)$ Gram matrix by the tuple $(G; \u, \v, t)$.
\end{enumerate}
\end{defn}

\begin{defn} Suppose $G$ is an $m \times m$ proper Gram matrix, $\u, \v \in \R^m$ and $t \in \R$. Define the $(m+2)$-point distance distribution by
\[
    N_m(G; \u, \v, t) \coloneqq \#\{(B, c, c') \in X^{m+2}: G(B, c, c') = (G; \u, \v, t)\},
\]
that is, the number of $(m+2)$-tuples $(B, c, c')$ with a specific Gram matrix. Also, we can define the $(m+1)$-point and $m$-point distance distributions
\begin{align*}
    N_{m-1}(G; \u) &\coloneqq \#\{(B, c) \in X^{m+1}: G(B, c) = (G; \u)\}, \\
    N_{m-2}(G) &\coloneqq \#\{B \in X^m: G(B) = G\}.
\end{align*}
\end{defn}

\begin{rmk} For $m=1$, $G = \begin{pmatrix}1 \end{pmatrix}$ and $N_1(G; \u, \v, t)$ is the three-point distance distribution $y(u, v, t)$; for $m=0$, $(m+2)$-tuples $(B, c, c')$ degenerated to pairs, so $N_0(G)$ is the two-point distance distribution $x(t)$ if $G = \begin{pmatrix}1 & t \\ t & 1\end{pmatrix}$; for $N=-1$, define $N_{-1} = N$.
\end{rmk}

{\bf Part 2.} The distance distribution $N_m$ is related to the cardinality $N$ by the following equation:
\[
    \sum_{G}\sum_{\u, \v \in A(X)^m}\sum_{t \in A(X)} N_m(G; \u, \v, t) = \frac{N!}{(N-m-2)!},
\]
in which $G$ is summed over all possible $m \times m$ proper Gram matrices. Furthermore, we have the following relations for $N_m$ and their degenerations.

\begin{prop}\label{prop: N}
Let $G$ be an $m \times m$ proper Gram matrix and $N_m$ be defined as above. Let $G_{(p)}$ be the $p$-th column of $G$.
\begin{enumerate}[label=(\alph*)]
    \item Suppose $u_p = v_q = 1$ for some $1 \leq p, q \leq m$. Then
    \[
        N_m(G; \u, \v, t) = \begin{dcases}
        N_{m-2}(G) & \text{ if } \u = G_{(p)}, \v = G_{(q)}, t = G_{pq}, \\
        0 & \text{ otherwise.}
        \end{dcases}
    \]
    \item Suppose $u_p = 1$ for some $1 \leq p \leq m$. Then
    \[
        N_m(G; \u, \v, t) = \begin{dcases}
        N_{m-1}(G; \v) & \text{ if } \u = G_{(p)}, t = v_p, \\
        0 & \text{ otherwise.}
        \end{dcases}
    \]
    Similarly, suppose $v_q = 1$ for some $1 \leq q \leq m$. Then
    \[
        N_m(G; \u, \v, t) = \begin{dcases}
        N_{m-1}(G; \u) & \text{ if } \v = G_{(q)}, t = u_q, \\
        0 & \text{ otherwise.}
        \end{dcases}
    \]
    \item Suppose $t = 1$. Then
    \[
        N_m(G; \u, \v, t) = \begin{dcases}
        N_{m-1}(G; \u) & \text{ if } \u = \v, \\
        0 & \text{ otherwise.}
        \end{dcases}
    \]
\end{enumerate}
\end{prop}
\begin{proof}
For (a), recall that $N_m(G; \u, \v, t)$ counts the tuples $(B, c, c')$ with
\[
    b_p \cdot b_q = G_{pq},\ b_p \cdot c = u_p,\ b_q \cdot c' = v_q,\ c \cdot c' = t. 
\]
If $u_p = v_q = 1$, then $c = b_p$ and $c' = b_q$. Therefore, the $(m+1)$-th column and the $p$-th column in $(G; \u, \v, t)$ must be the same, i.e., $\u = G_{(p)}$ and $t = v_p$. Similarly, $\v = G_{(q)}$ and $t = u_q = G_{pq}$. Hence $N_m(G; \u, \v, t) > 0$ only if $\u = G_{(p)}$, $\v = G{(q)}$ and $t = G_{pq}$; also, the Gram matrix $(G; \u, \v, t)$ of $(B, c, c')$ degenerates to the Gram matrix $G$ of $B$. \\
For (b), if $u_p = 1$ then $c = b_p$, and the Gram matrix of $(B, c, c')$ degenerates to that of $(B, c')$; if $v_q = 1$ then $c' = b_p'$, and the Gram matrix degenerates to that of $(B, c)$. For (c), if $t = 1$ then $c = c'$, and the Gram matrix degenerates to that of $(B, c)$.
\end{proof}

{\bf Part 3.} \citet{musin2013multivariate} proved a multi-point generalization for the semidefinite constraints~\eqref{eq: 3sdp_doublesum}, which we use to formulate the constraints for the multi-point distance distribution $N_m$.

\begin{defn}[\citet{musin2013multivariate}]
    Let $G$ be an $m \times m$ proper Gram matrix, $\u, \v \in \R^m$ and $t \in \R$. Define the {\bf multivariate Gegenbauer polynomials} by
    \[
        Q^{d, m}_k(G; \u, \v, t) \coloneqq (1 - \u G^{-1}\u^{\T})^{k/2}(1 - \v G^{-1}\v^{\T})^{k/2} P^{d-m}_k\left(\frac{t- \u G^{-1}\v^{\T}}{\sqrt{(1 - \u G^{-1}\u^{\T})(1 - \v G^{-1}\v^{\T})}}\right).
    \]
    Let $\mathcal{B}_n(\u)$ be the column vector collecting all monomials of $\{u_1, \dots, u_m\}$ with degree at most $n$. Define the matrices $Y^{d, m}_{k, n}$ by
    \[
        Y^{d, m}_{k, n}(G; \u, \v, t) \coloneqq Q^{d, m}_k(G; \u, \v, t)\mathcal{B}_n(u)\mathcal{B}_n(v)^{\T}.
    \]
\end{defn}
The order of $Y^{d, m}_{k, n}$ is $\binom{n+m}{m}$, and in which every entry is a polynomial of $\u, \v, t$ in degree $k$. When $m=1$, the Gram matrix $G$ is the identity matrix of order $1$, and the above definitions of $Q^{d, m}_k$ and $Y^{d, m}_{k, n}$ agree with $Q^d_k$ and $Y^d_{k, n}$ defined by~\citet{bachoc2008new}.

Fix a subset $B = \{b_1, \dots, b_m\} \subseteq X$ of linearly independent vectors. Let $G = G(B)$ be the Gram matrix of $B$. \citet{musin2013multivariate} proved that
\begin{equation}\label{eq: msdp_doublesum}
    \sum_{c, c' \in X} Y^{d, m}_{k, n}(G; B^{\T} \cdot c, B^{\T} \cdot c', c \cdot c') \succeq 0.
\end{equation}

By summing $B$ in~\eqref{eq: msdp_doublesum} over all subsets of $X$ with fixed proper Gram matrix $G$, we have the following constraints about $N_m$.
\begin{thm}\label{thm: msdp_original}
    For a spherical set $X$ in $\S^{d-1}$ with inner product set $A$,
    \[
        \sum_{G}\sum_{\u, \v \in A^m}\sum_{t \in A} N_m(G; \u, \v, t) = \frac{N!}{(N-m-2)!},
    \]
    in which $G$ is summed over all possible $m \times m$ proper Gram matrices. Let $A' = A \cup \{1\}$. For a fixed $m \times m$ proper Gram matrix $G$,
    \[
        \sum_{\u, \v \in (A')^m}\sum_{t \in A'} N_m(G; \u, \v, t)Y^{d, m}_{k, n}(G; \u, \v, t) \succeq 0 \text{ for } k, n \geq 0.
    \]
\end{thm}
Note that there is one semidefinite constraint for each proper Gram matrix $G$. For example, for a spherical $s$-distance set $X$ with $-1 \notin A(X)$, there are $s$ possible $2 \times 2$ proper Gram matrices $G$, and accordingly $s$ different semidefinite constraints for the $4$-point distance distribution $N_2(G; \u, \v, t)$.

{\bf Part 3 (alternative version).} Similar to the relation of Theorem~\ref{thm: 3sdp_original} and Theorem~\ref{thm: 3sdp_alternative}, there is also an alternative version of Theorem~\ref{thm: msdp_original}. We call it the {\bf alternative multi-point semidefinite constraint}.

\begin{thm}[Alternative multi-point semidefinite constraint]\label{thm: msdp_alternative}
    For a spherical $s$-distance set $X$ in $\S^{d-1}$ with inner product set $A$,
    \[
        \sum_{G}\sum_{\u, \v \in A^m}\sum_{t \in A} N_m(G; \u, \v, t) = \frac{N!}{(N-m-2)!},
    \]
    in which $G$ is summed over all possible $m \times m$ proper Gram matrices. Let $A' = A \cup \{1\}$ and $\{\e_{\u}: \u \in (A')^m\}$ be the standard basis of $\R^{(s+1)^m}$. Then for a fixed $m \times m$ proper Gram matrix $G$,
    \[
        \mathcal{Q}^{d, m}_k(G; X) \coloneqq \sum_{\u, \v \in (A')^m}\sum_{t \in A'} N_m(G; \u, \v, t)Q^{d, m}_k(G; \u, \v, t)\e_\u\e_\v^{\T} \succeq 0 \text{ for } k \geq 0.
    \]
\end{thm}
\begin{proof}
The proof is similar to the proof of Theorem~\ref{thm: 3sdp_alternative}. By Theorem~\ref{thm: msdp_original},
\[
    \sum_{\u, \v \in (A')^m}\sum_{t \in A'} N_m(G; \u, \v, t)Y^{d, m}_{k, n}(G; \u, \v, t) \succeq 0,
\]
i.e.,
\[
    \sum_{\u, \v \in (A')^m}\sum_{t \in A'} N_m(G; \u, \v, t)Q^{d, m}_k(G; \u, \v, t)\mathcal{B}_n(\u)\mathcal{B}_n(\v)^{\T} \succeq 0.
\]
Hence
\begin{equation}\label{eq: msdp_altform_1}
    \sum_{\u, \v \in (A')^m}\sum_{t \in A'} N_m(G; \u, \v, t)Q^{d, m}_k(G; \u, \v, t)f(\u)f(\v) \geq 0
\end{equation}
for all $m$-variable polynomials $f$, since the degree $n$ is arbitrary.

Let $\textstyle \mathbf{f} = \sum_{\u \in (A')^m} f(\u)\e_{\u} \in \R^{(m+1)^s}$. Then inequality~\eqref{eq: msdp_altform_1} becomes
\[
    \mathbf{f}^{\T}\Big(\sum_{\u, \v \in (A')^m}\sum_{t \in A'}N_m(G; \u, \v, t)Q^{d, m}_k(G; \u, \v, t)\e_{\u}\e_{\v}^{\T}\Big)\mathbf{f} = \mathbf{f}^{\T}\SDP^{d, m}_k(G; X)\mathbf{f} \geq 0
\]
for all possible $\mathbf{f}$. Since the possible range of $\mathbf{f}$ is as large as $\R^{(m+1)^s}$, $\SDP^{d, m}_k(G; X)$ is positive semidefinite.
\end{proof}

For a spherical $s$-distance set $X$, the matrices $\SDP^{d, m}_k(G; X)$ in the above semidefinite constraints are $(s+1)^m \times (s+1)^m$ matrices, in which every entry contains $(s+1)$ terms corresponding to $(s+1)$ values $t \in A'$. When $k = 0$, the polynomial $Q^{d, m}_k$ is constant; when $k > 0$, the polynomial $Q^{d, m}_k(G; \u, \v, t)$ vanishes for some specific input. In either case, the matrix $\SDP^{d, m}_k(G; X)$ has a cleaner format. As the following indicates, $\SDP^{d, m}_k(G; X)$ can be represented as a $(s^m+1) \times (s^m+1)$ matrix if $k = 0$, and a $s^m \times s^m$ matrix if $k \geq 1$.

\begin{cor}\label{cor: msdp_alternative}
Let $G_{(p)}$ be the $p$-th column of $G$.
\begin{enumerate}[label=(\alph*)]
    \item For $k=0$, let $\e_{\1} = \e_{G_{(1)}} + \cdots + \e_{G_{(m)}}$. Then
    \begin{align*}
        \SDP^{d, m}_0(G; X) &= N_{m-2}(G)\e_{\1}\e_{\1}^{\T} \\
        &+ \sum_{\u \in A^m}N_{m-1}(G; \u)(\e_{\1}\e_{\u}^{\T} + \e_{\u}\e_{\1}^{\T} + \e_{\u}\e_{\u}^{\T}) \\
        &+ \sum_{\u, \v \in A^m}\sum_{t \in A} N_m(G; \u, \v, t)\e_{\u}\e_{\v}^{\T}.
    \end{align*}
    \item For $k \geq 1$,
    \begin{align*}
        \SDP^{d, m}_k(G; X) &= \sum_{\u \in A^m}N_{m-1}(G; \u)Q^{d, m}_k(G; \u, \u, 1)\e_{\u}\e_{\u}^{\T} \\
        &+ \sum_{\u, \v \in A^m}\sum_{t \in A} N_m(G; \u, \v, t)Q^{d, m}_k(G; \u, \v, t)\e_{\u}\e_{\v}^{\T}.
    \end{align*}
\end{enumerate}
\end{cor}
\begin{proof}
By Proposition~\ref{prop: N}, if at least one of $\{u_1, \dots, u_m, v_1, \dots, v_m\}$ equals $1$, then there is at most one $N_m(G; \u, \v, t)$ being positive among all $t \in A'$; if $t = 1$, then $N_m(G; \u, \v, t) > 0$ only when $\u = \v$. After specifying all the special cases, the matrix $\SDP^{d, m}_k(G; X)$ becomes
\begin{align*}
    \SDP^{d, m}_k(G; X) &= \sum_{1 \leq p, q \leq m}N_{m-2}(G)Q^{d, m}_k(G; G_{(p)}, G_{(q)}, G_{pq})\e_{G_{(p)}}\e_{G_{(q)}}^{\T} \\
    &+ \sum_{1 \leq p \leq m}\sum_{\v \in A^m}N_{m-1}(G; \v)Q^{d, m}_k(G; G_{(p)}, \v, v_p)\e_{G_{(p)}}\e_{\v}^{\T} \\
    &+ \sum_{\u \in A^m}\sum_{1 \leq q \leq m}N_{m-1}(G; \u)Q^{d, m}_k(G; \u, G_{(q)}, u_q)\e_{\u}\e_{G_{(q)}}^{\T} \\
    &+ \sum_{\u \in A^m}N_{m-1}(G; \u)Q^{d, m}_k(G; \u, \u, 1)\e_{\u}\e_{\u}^{\T} \\
    &+ \sum_{\u, \v \in A^m}\sum_{t \in A} N_m(G; \u, \v, t)Q^{d, m}_k(G; \u, \v, t)\e_{\u}\e_{\v}^{\T}.
\end{align*}

The equation in (a) follows from the fact that $Q^{d, m}_0(G; \u, \v, t) \equiv 1$. For (b), we claim that all $Q^{d, m}_k(G; \u, \v, t)$ in the first three lines vanish since $Q^{d, m}_k(G; \u, \v, t)$ is a polynomial in $t - \u G^{-1}\v^{\T}$ and $(1 - \u G^{-1}\u^{\T})(1 - \v G^{-1}\v^{\T})$. For example, with input $(G; \u, \v, t) = (G; G_{(p)}, \v, v_p)$ in the second line,
\[
    t - \u G^{-1}\v^{\T} = v_p - G_{(p)}G^{-1}\v^{\T} = v_p - \e_p\v^{\T} = v_p - v_p = 0
\]
and
\[
    1 - \u G^{-1}\u^{\T} = 1 - G_{(p)}G^{-1}G_{(p)}^{\T} = 1 - G_{pp} = 0,
\]
so $Q^{d, m}_k(G; G_{(p)}, \v, v_p) = 0$.
\end{proof}

{\bf Part 4.} One can establish a semidefinite program using Theorem~\ref{thm: msdp_original} or using Theorem~\ref{thm: msdp_alternative} to bound the cardinality $N$.
\section{Switching reduction}
\label{sec: switching}

In this section, we give a simplification of the semidefinite constraints on the regime that $X$ is a spherical projection of a set of equiangular lines, which we call by {\bf switching reduction}.

The simplification is given by considering the mutual relations of different spherical projections. We say that two Gram matrices $G_1$, $G_2$ are {\bf switching equivalent}, if $G_2$ is given by changing the sign of some columns and the corresponding rows in $G_1$. The Gram matrices of different spherical projections of the same set of equiangular lines are switching equivalent: let $X = \{c_1, \dots, c_N\}$ be a spherical projection with Gram matrix $G = G(X)$. Any possible spherical projection is of the form
\[
    \{\eps_1c_1, \dots, \eps_Nc_N\}, \eps_i \in \{\pm 1\}.
\]
Let $\Lambda = \text{diag}(\eps_1, \dots, \eps_N)$ be the diagonal matrix with diagonal entries $\eps_1, \dots, \eps_N$. Then the Gram matrix $G_{\Lambda}$ of the corresponding spherical projection is
\[
    G_{\Lambda} = \begin{pmatrix}\eps_ic_i \cdot \eps_jc_j\end{pmatrix}_{1 \leq i, j \leq N} = \Lambda G \Lambda.
\]
That is, $G_{\Lambda}$ is given by changing the sign of the column $i$ and the corresponding row $i$ with $\eps_i = -1$. Then $G_{\Lambda}$ is switching equivalent to $G$.

In the following, we demonstrate the modification of the semidefinite constraints after switching reduction.

{\bf Part 1.} We can define the distance distribution $N_m(G; \u, \v, t)$ for a spherical projection $X$ of a set of equiangular lines. Meanwhile, we can never define $N_m(G; \u, \v, t)$ for the set of equiangular lines itself, since the number of $(m+2)$-tuples with a fixed Gram matrix depends on the chosen spherical projection. However, the sum of $N_m(G; \u, \v, t)$ over a switching class of the Gram matrix $(G; \u, \v, t)$ is independent of the choice of spherical projections.

We define the {\bf distance distribution of switching classes}, which do not rely on the spherical projections. We use square brackets to represent the sums over switching classes.
\begin{defn}\label{defn: Nm[]}
Suppose $G$ is an $m \times m$ proper Gram matrix and $u_i, v_j, t \in \{\pm\a\}$. Define the distance distribution of the switching class $[G; \u, \v, t]$ by
\[
    N_m[G; \u, \v, t] \coloneqq \sum_{(G'; \u', \v', t') \sim (G; \u, \v, t)} N_m(G'; \u', \v', t'),
\]
in which $\sim$ means switching equivalence. The value $N_m[G; \u, \v, t]$ is the number of $(m+2)$-tuples $(B, c, c')$ with Gram matrix switching equivalent to $(G; \u, \v, t)$. We also define
\begin{align*}
    N_{m-1}[G; \u] &= \sum_{(G'; \u') \sim (G; \u)} N_{m-1}(G'; \u'), \\
    N_{m-2}[G] &= \sum_{G' \sim G} N_{m-2}(G').
\end{align*}
\end{defn}

Algebraically, $(G'; \u', \v', t')$ is switching equivalent to $(G; \u, \v, t)$ if and only if
\[
    (G'; \u', \v', t') = \Lambda_0 (G; \u, \v, t) \Lambda_0
\]
for some $(m+2) \times (m+2)$ diagonal matrix $\Lambda_0 = \text{diag}(\eps_1, \dots, \eps_{m+2})$ with diagonal entries $\eps_1, \dots, \eps_{m+2} \in \{\pm 1\}$. Write $\Lambda = \text{diag}(\eps_1, \dots, \eps_m)$, we have
\begin{align*}
    \begin{pmatrix} G' & \u' & \v' \\
    (\u')^{\T} & 1 & t' \\
    (\v')^{\T} & t' & 1 \end{pmatrix} &= \begin{pmatrix} \Lambda & 0 & 0 \\ 0 & \eps_{m+1} & 0 \\ 0 & 0 & \eps_{m+2} \end{pmatrix} \begin{pmatrix} G & \u & \v \\
    \u^{\T} & 1 & t \\
    \v^{\T} & t & 1 \end{pmatrix} \begin{pmatrix} \Lambda & 0 & 0 \\ 0 & \eps_{m+1} & 0 \\ 0 & 0 & \eps_{m+2} \end{pmatrix} \\
    &= \begin{pmatrix}
        \Lambda G\Lambda & \eps_{m+1}\Lambda\u & \eps_{m+2}\Lambda\v \\
        \eps_{m+1}\u^{\T}\Lambda & \eps_{m+1}^2 & \eps_{m+1}\eps_{m+2}t \\
        \eps_{m+2}\u^{\T}\Lambda & \eps_{m+2}\eps_{m+1}t & \eps_{m+2}^2
    \end{pmatrix},
\end{align*}
therefore
\[
    G' = \Lambda G\Lambda,\ \u' = \eps_{m+1}\Lambda\u,\ \v' = \eps_{m+2}\Lambda\v,\ t' = \eps_{m+1}\eps_{m+2} t.
\]

\begin{prop}\label{prop: N_sum}
Let $G$ be an $m \times m$ proper Gram matrix and $u_i, v_j, t \in \{\pm\a\}$. Then
\begin{enumerate}[label=(\alph*)]
    \item $N_m[G; \u, \v, t] = \frac{1}{2}\sum_{\Lambda}\sum_{\eps_{m+1}, \eps_{m+2} \in \{\pm 1\}} N_m(\Lambda G \Lambda; \eps_{m+1}\Lambda \u, \eps_{m+2}\Lambda\v, \eps_{m+1}\eps_{m+2}t)$.
    \item $N_{m-1}[G; \u] = \frac{1}{2}\sum_{\Lambda}\sum_{\eps_{m+1} \in \{\pm 1\}} N_{m-1}(\Lambda G\Lambda; \eps_{m+1}\Lambda\u)$.
    \item $N_{m-2}[G] = \frac{1}{2}\sum_{\Lambda} N_{m-2}(\Lambda G\Lambda)$.
\end{enumerate}
In all cases, $\Lambda$ is summed over all $m \times m$ diagonal matrices with diagonal entries $\eps_1, \dots, \eps_m \in \{\pm 1\}$.
\end{prop}

There is a multiple $1/2$ in Proposition~\ref{prop: N_sum}, since two diagonal matrices define the same Gram matrix if and only if the matrices are opposite.

{\bf Part 2.} The distance distribution $N_m[G; \u, \v, t]$ of switching classes are related to the cardinality $N$ by the following equations.
\begin{prop}\label{prop: N_sum2}
Let $\mathcal{G}$ be a collection of representatives of switching classes of $m \times m$ proper Gram matrices. Fix $u_1 = v_1 = \a$, and let $u_p, v_q, t \in \{\pm\a\}$ for $2 \leq p, q \leq m$. Then
\begin{align*}
    \sum_{G \in \mathcal{G}}\sum_{u_1 = v_1 = \a}\sum_{t \in \{\pm\a\}}N_m[G; \u, \v, t] &= \frac{N!}{(N-m-2)!}, \\
    \sum_{G \in \mathcal{G}}\sum_{u_1 = \a} N_{m-1}[G; \u] &= \frac{N!}{(N-m-1)!}, \\
    \sum_{G \in \mathcal{G}}N_{m-2}[G] &= \frac{N!}{(N-m)!}.
\end{align*}
\end{prop}
\begin{proof}
The right hand side of the first equation is the sum of $N_m(G; \u, \v, t)$ over all proper $(m+2) \times (m+2)$ Gram matrices $N_m(G; \u, \v, t)$. For any $(G'; \u', \v', t')$, there is exactly one $m \times m$ Gram matrix $G = \Lambda G'\Lambda$ which is switching equivalent to $G'$. Therefore, the number $N_m(G'; \u', \v', t')$ is counted in
\[
    N_m[G; \eps\Lambda\u', \eps'\Lambda\v', \eps\eps't']
\]
for $\eps, \eps' \in \{\pm 1\}$. Only one of the four terms is counted in the left hand side since $u_1$ and $v_1$ are restricted to $+\a$. Therefore, the both sides are the same. By similar arguments the other two equations are true.
\end{proof}

There are several duplications of values in the distance distribution $N_m(G; \u, \v, t)$, as
\[
    N_m(G; \u, \v, t) = N_m(G'; \u', \v', t')
\]
if the matrices $(G; \u, \v, t)$, $(G'; \u', \v', t')$ are identical up to permutations. We extend this fact to the duplication of values in the distance distribution of switching classes.
\begin{prop}\label{prop: N_permutation}
If $(G'; \u', \v', t')$ is switching equivalent to $P(G; \u, \v, t)P^{\T}$ for some permutation matrix $P$, then $N_m[G'; \u', \v', t'] = N_m[G; \u, \v, t]$. Similarly, if $(G'; \u')$ is switching equivalent to $P(G; \u)P^T$ for some permutation matrix $P$, then $N_{m-1}[G'; \u'] = N_{m-1}[G; \u]$; if $G'$ is switching equivalent to $PGP^T$ for some permutation matrix $P$, then $N_{m-2}[G'] = N_{m-2}[G]$.
\end{prop}

\begin{eg} \label{eg: switching}
For the two-point distance distribution, the two $2 \times 2$ proper Gram matrices are switching equivalent to each other, so $N_0[G] = N(N-1)$ for either $G$. For the three-point distance distribution, note that any $3 \times 3$ proper Gram matrix $(1; u, v, t)$ is switching equivalent to either $\begin{pmatrix}1 & \a & \a \\ \a & 1 & \a \\ \a & \a & 1\end{pmatrix}$ or $\begin{pmatrix}1 & \a & \a \\ \a & 1 & -\a \\ \a & -\a & 1\end{pmatrix}$, so there are $2$ switching classes in total. 
Denote
\begin{align*}
    y_1 &\coloneqq N_1[1; \a, \a, \a] = y(\a, \a, \a) + y(\a, -\a, -\a) + y(-\a, \a, -\a) + y(-\a, -\a, \a), \\
    y_2 &\coloneqq N_1[1; \a, \a, -\a] = y(\a, \a, -\a) + y(\a, -\a, \a) + y(-\a, \a, \a) + y(-\a, -\a, -\a).
\end{align*}
By Proposition~\ref{prop: N_sum2}, $y_1 + y_2 = N!/(N-3)! = N(N-1)(N-2)$.
For the four-point distance distribution $N_2[G; \u, \v, t]$, note that any $4 \times 4$ proper Gram matrix $(G'; \u', \v', t')$ is switching equivalent to a Gram matrix of the form
\[
    \begin{pmatrix}
        G & \u & \v \\
        \u^{\T} & 1 & t \\
        \v^{\T} & t & 1
    \end{pmatrix} = \begin{pmatrix}
        1 & \a & \a & \a \\
        \a & 1 & u_2 & v_2 \\
        \a & u_2 & 1 & t \\
        \a & v_2 & t & 1
    \end{pmatrix},\ u_2, v_2, t \in \{\pm\a\},
\]
so there are $2^3 = 8$ switching classes in total. Furthermore, the six Gram matrices
\begin{align*}
    \begin{pmatrix}
        1 & \a & \a & \a \\
        \a & 1 & \a & \a \\
        \a & \a & 1 & -\a \\
        \a & \a & -\a & 1
    \end{pmatrix}, \begin{pmatrix}
        1 & \a & \a & \a \\
        \a & 1 & \a & -\a \\
        \a & \a & 1 & \a \\
        \a & -\a & \a & 1
    \end{pmatrix}, \begin{pmatrix}
        1 & \a & \a & \a \\
        \a & 1 & \a & -\a \\
        \a & \a & 1 & -\a \\
        \a & -\a & -\a & 1
    \end{pmatrix}, \\
    \begin{pmatrix}
        1 & \a & \a & \a \\
        \a & 1 & -\a & \a \\
        \a & -\a & 1 & \a \\
        \a & \a & \a & 1
    \end{pmatrix}, \begin{pmatrix}
        1 & \a & \a & \a \\
        \a & 1 & -\a & \a \\
        \a & -\a & 1 & -\a \\
        \a & \a & -\a & 1
    \end{pmatrix},
    \begin{pmatrix}
        1 & \a & \a & \a \\
        \a & 1 & -\a & -\a \\
        \a & -\a & 1 & \a \\
        \a & -\a & \a & 1
    \end{pmatrix}
\end{align*}
are identical up to switching and permutations. Let $\u_1 = \begin{pmatrix}\a & \a\end{pmatrix}^{\T}$ and $\u_2 = \begin{pmatrix}\a & -\a\end{pmatrix}^{\T}$. By Proposition~\ref{prop: N_permutation},
\begin{align*}
    & N_2[G; \u_1, \u_1, -\a] = N_2[G; \u_1, \u_2, \a] = N_2[G; \u_1, \u_2, -\a] \\
    =& N_2[G; \u_2, \u_1, \a] = N_2[G; \u_2, \u_1, -\a] = N_2[G; \u_2, \u_2, \a].
\end{align*}
Denote $z_1 \coloneqq N_2[G; \u_1, \u_1, \a]$, $z_2 \coloneqq N_2[G; \u_1, \u_1, -\a]$ and $z_3 \coloneqq N_2[G; \u_2, \u_2, -\a]$. By Proposition~\ref{prop: N_sum2}, we have $z_1 + 6z_2 + z_3 = N(N-1)(N-2)(N-3)$.
\end{eg}

\begin{rmk}
In general, there are $2^{\binom{m+2}{2}}$ different $(m+2) \times (m+2)$ Gram matrices and $2^{\binom{m+1}{2}}$ switching classes. In the spirit of Proposition~\ref{prop: N_permutation}, there are at most $a(m+2)$ different values of the distance distribution $N_m[G; \u, \v, t]$ of switching classes, where $a(m)$ is the number of Seidel matrices of order $m$. The first seven terms for $a(m)$ is $1, 1, 2, 3, 7, 16, 54$; see Sequence A002854 in The On-Line Encyclopedia of Integer Sequences~\cite{oeis} for detail. Therefore, the number of variables of a semidefinite program developed by $(m+2)$-point semidefinite constraints is decreased to $a(m+2)$ after switching reduction.
\end{rmk}

{\bf Part 3.} The reduction in the alternative semidefinite constraints $\SDP^{d, m}_k(G; X)$ is due to the {\bf switching property} of the polynomial $Q^{d, m}_k(G; \u, \v, t)$.
\begin{prop}\label{prop: Q_switching}
If $(G'; \u', \v', t)$ is switching equivalent to $(G; \u, \v, t)$ by the relation
\[
    G' = \Lambda G\Lambda,\ \u' = \eps_{m+1}\Lambda\u,\ \v' = \eps_{m+2}\Lambda\v,\ t' = \eps_{m+1}\eps_{m+2} t,
\]
then $Q^{d, m}_k(G'; \u', \v', t') = (\eps_{m+1}\eps_{m+2})^kQ^{d, m}_k(G; \u, \v, t)$ for any $k \geq 0$.
\end{prop}
\begin{proof}
Since
\[
    t' - \u'(G')^{-1}(\v')^{\T} = \eps_{m+1}\eps_{m+2}(t - \u G^{-1}\v^{\T})
\]
and
\[
    (1 - \u'(G')^{-1}(\u')^{\T})(1 - \v'(G')^{-1}(\v')^{\T}) = (1 - \u G^{-1}\u^{\T})(1 - \v G^{-1}\v^{\T}),
\]
the only difference in $Q^{d, m}_k(G; \u, \v, t)$ and $Q^{d, m}_k(G'; \u', \v', t')$ is the input of $P^{d-m}_k$ in a multiple $\eps_{m+1}\eps_{m+2}$. Meanwhile, $P^{d-m}_k$ is an odd function when $k$ is odd and an even function when $k$ is even, so the two values differ by a multiple of $(\eps_{m+1}\eps_{m+2})^k$.
\end{proof}

Due to Proposition~\ref{prop: Q_switching}, we can safely sum over switching classes in the alternative semidefinite constraints $\SDP^{d, m}_k(G; X)$, since the coefficients $Q^{d, m}_k$ for inputs $(G; \u, \v, t)$ in the same switching class are identical. We also use the square bracket to represent the sum of the matrices over a switching class.

\begin{thm}\label{thm: msdp_reduction}
For a spherical projection $X$ of a set of equiangular lines in $\R^d$ with inner product set $A = \{\pm\a\}$, fix $u_1, v_1 = \a$, and let $u_p, v_q, t \in \{\pm\a\}$ for $2 \leq p, q \leq m$. For a fixed $m \times m$ proper Gram matrix $G$,
\begin{align*}
    \SDP^{d, m}_0[G; X] &\coloneqq N_{m-2}[G]\e_{\1}\e_{\1}^{\T} \\
    &+ \sum_{u_1 = \a}N_{m-1}[G; \u](\e_{\1}\e_{\u} + \e_{\u}\e_{\1}^{\T} + \e_{\u}\e_{\u}^{\T}) \\
    &+ \sum_{u_1 = v_1 = \a}\left(N_m[G; \u, \v, \a] + N_m[G; \u, \v, -\a]\right)\e_{\u}\e_{\v}^{\T} \succeq 0.
\end{align*}
For $k \geq 1$,
\begin{align*}
    \SDP^{d, m}_k[G; X] &\coloneqq \sum_{u_1 = \a} N_{m-1}[G; \u]Q^{d, m}_k(G; \u, \u, 1)\e_{\u}\e_{\u}^{\T} \\
    &+ \sum_{u_1 = v_1 = \a}\left(N_m[G; \u, \v, \a]Q^{d, m}_k(G; \u, \v, \a) + N_m[G; \u, \v, -\a]Q^{d, m}_k(G; \u, \v, -\a)\right)\e_{\u}\e_{\v}^{\T} \succeq 0.
\end{align*}
\end{thm}
\begin{proof}
We first prove the statement for $k=0$. By Corollary~\ref{cor: msdp_alternative}(a), for any $G' = \Lambda G \Lambda$ where $\Lambda$ is a diagonal matrix with diagonal entries $\pm 1$,
\begin{align}
    \SDP^{d, m}_0(G'; X) &= N_{m-2}(\Lambda G\Lambda)\e_{\1}\e_{\1}^{\T} + \sum_{\u \in A^m}N_{m-1}(\Lambda G\Lambda; \Lambda\u)(\e_{\1}\e_{\Lambda\u}^{\T} + \e_{\Lambda\u}\e_{\1}^{\T} + \e_{\Lambda\u}\e_{\Lambda\u}^{\T}) \nonumber \\
    &+ \sum_{\u, \v \in A^m}\left(N_m(\Lambda G\Lambda; \Lambda\u, \Lambda\v, \a) + N_m(\Lambda G\Lambda; \Lambda\u, \Lambda\v, -\a)\right)\e_{\Lambda\u}\e_{\Lambda\v}^{\T} \succeq 0. \label{eq: msdp_reduction_0}
\end{align}
Let $S$ be the linear transformation defined by $S(\e_{\1}) = \e_{\1}$, $S(\e_{\Lambda\u}) = \e_{\u}$ for $u_1 = \a$, and $S(\e_{\Lambda\u}) = \e_{-\u}$ for $u_1 = -\a$. By conjugating $S$ on \eqref{eq: msdp_reduction_0}, we have
\begin{align*}
    & N_{m-2}(\Lambda G\Lambda)\e_{\1}\e_{\1}^{\T} + \sum_{\u \in A^m}N_{m-1}(\Lambda G\Lambda; \Lambda\u)(\e_{\1}(S\e_{\Lambda\u})^{\T} + (S\e_{\Lambda\u})\e_{\1}^{\T} + (S\e_{\Lambda\u})(S\e_{\Lambda\u})^{\T}) \\
    +& \sum_{\u, \v \in A^m}\left(N_m(\Lambda G\Lambda; \Lambda\a_{\u}, \Lambda\a_{\v}, \a) + N_m(\Lambda G\Lambda; \Lambda\a_{\u}, \Lambda\a_{\v}, -\a)\right)(S\e_{\Lambda\u})(S\e_{\Lambda\v})^{\T} \succeq 0,
\end{align*}
therefore
\begin{align}
    & N_{m-2}(\Lambda G\Lambda)\e_{\1}\e_{\1}^{\T} + \sum_{u_1 = \a}\sum_{\eps \in \{\pm 1\}}N_{m-1}(\Lambda G\Lambda; \eps\Lambda\u)(\e_{\1}\e_{\u}^{\T} + \e_{\u}\e_{\1}^{\T} + \e_{\u}\e_{\u}^{\T}) \nonumber \\
    +& \sum_{u_1 = v_1 = \a}\sum_{\eps, \eps' \in \{\pm 1\}}\left(N_m(\Lambda G\Lambda; \eps\Lambda\u, \eps'\Lambda\v, \eps\eps'\a) + N_m(\Lambda G\Lambda; \eps\Lambda\u, \eps'\Lambda\v, -\eps\eps'\a)\right)\e_{\u}\e_{\v}^{\T} \succeq 0. \label{eq: msdp_reduction_1}
\end{align}
Summing over all $\Lambda$ in \eqref{eq: msdp_reduction_1} gives $2\SDP^{d, m}_0[G; X] \succeq 0$.

Now we fix $k \geq 1$. By Corollary~\ref{cor: msdp_alternative}(b), for any $G' = \Lambda G \Lambda$,
\begin{align}
    & \SDP^{d, m}_k(G'; X) = \sum_{\u \in A^m}N_{m-1}(\Lambda G\Lambda; \Lambda\u)Q^{d, m}_k(\Lambda G\Lambda; \Lambda\u, \Lambda\u, 1)\e_{\Lambda\u}\e_{\Lambda\u}^{\T} \nonumber \\
    +& \sum_{\u, \v \in A^m}\Big(N_m(\Lambda G\Lambda; \Lambda\u, \Lambda\v, \a)Q^{d, m}_k(\Lambda G\Lambda; \Lambda\u, \Lambda\v, \a) + N_m(\Lambda G\Lambda; \Lambda\u, \Lambda\v, -\a)Q^{d, m}_k(\Lambda G\Lambda; \Lambda\u, \Lambda\v, -\a)\Big)\e_{\Lambda\u}\e_{\Lambda\v}^{\T} \succeq 0. \label{eq: msdp_reduction_2}
\end{align}
Let $S$ be the linear transformation defined by $S(\e_{\Lambda\u}) = \e_{\u}$ for $u_1 = \a$, and $S(\e_{\Lambda\u}) = (-1)^k\e_{-\u}$ for $u_1 = -\a$. By conjugating $S$ on \eqref{eq: msdp_reduction_2} and applying Proposition~\ref{prop: Q_switching}, the first term in~\eqref{eq: msdp_reduction_2} becomes
\begin{align*}
    & \sum_{u_1 = \a} \Big(N_{m-1}(\Lambda G\Lambda; \Lambda\u)Q^{d, m}_k(\Lambda G\Lambda; \Lambda\u, \Lambda\u, 1) + (-1)^{2k}N_{m-1}(\Lambda G\Lambda; -\Lambda\u)Q^{d, m}_k(\Lambda G\Lambda; -\Lambda\u, -\Lambda\u, 1)\Big)\e_{\u}\e_{\u}^{\T} \\
    =& \sum_{u_1 = \a} \Big(N_{m-1}(\Lambda G\Lambda; \Lambda\u) + N_{m-1}(\Lambda G\Lambda; -\Lambda\u)\Big)Q^{d, m}_k(G; \u, \u, 1)\e_{\u}\e_{\u}^{\T}
\end{align*}
and the second term becomes
\begin{align*}
    &\sum_{u_1 = v_1 = \a} \Bigg[ \left(\sum_{\eps, \eps' \in \{\pm 1\}} N_m(\Lambda G\Lambda; \eps\Lambda\u, \eps\Lambda\v, \eps\eps'\a)\right)Q^{d, m}_k(G; \u, \v, \a) \\
    &= \left(\sum_{\eps, \eps' \in \{\pm 1\}} N_m(\Lambda G\Lambda; \eps\Lambda\u, \eps\Lambda\v, -\eps\eps'\a)\right)Q^{d, m}_k(G; \u, \v, -\a) \Bigg] \e_{\u}\e_{\v}^{\T}.
\end{align*}
Therefore, summing over all $\Lambda$ in \eqref{eq: msdp_reduction_2} gives $2\SDP^{d, m}_k[G; X] \succeq 0$.
\end{proof}

\begin{rmk}
After switching reduction, the matrices $\SDP^{d, m}_k[G; X]$ have order $2^{m-1}+1$ if $k=0$, and order $2^{m-1}$ if $k \geq 1$. The order is halved in contrast to Corollary~\ref{cor: msdp_alternative}. Furthermore, if $G'$ is switching equivalent to $PGP^{\T}$ for some permutation matrix $P$, then $\SDP^{d, m}_k[G; X]$ and $\SDP^{d, m}_k[G'; X]$ are identical up to relabelling of basis. Therefore, the number of semidefinite constraints of a semidefinite program developed by $(m+2)$-point semidefinite constraints is decreased to $a(m)$ after switching reduction.
\end{rmk}
\section{Four-point semidefinite bound for equiangular lines}
\label{sec: 4sdp}

\citet{yu2017new} proved the three-point semidefinite bound, which states that $N_{1/a}(d) \leq (a^2-1)(a^2-2)/2$ for $d \leq D_3(a) = 3a^2-16$. The bound was proved by using the three-point semidefinite programming method based on Theorem~\ref{thm: 3sdp_original}. Furthermore,~\citet{glazyrin2018upper} proved that the constructions attaining this bound must lie in an $(a^2-2)$-dimensional subspace. We prove the generalizations of these results, namely, Theorem~\ref{thm: 4sdp_main_mirror} and Theorem~\ref{thm: maximum_mirror}, in this section.

\subsection{Proof of Theorem~\ref{thm: 4sdp_main_mirror}}
We state Theorem~\ref{thm: 4sdp_main_mirror} as follows.
\begin{thm}\label{thm: 4sdp_main}
Let $a \geq 3$ be an odd integer. Suppose $d \leq \lfloor D_4(a)\rfloor$, where $D_4(a)$ is the maximum root of the equation
\begin{align*}
    g_a(x) &\coloneqq (-7a^{14} - 122a^{12} - 342a^{10} + 2776a^8 + 7049a^6 - 17238a^4 - 22932a^2 - 6048)x^4 \\
    &+ 12(4a^{16} + 21a^{14} - 227a^{12} - 46a^{10} + 3338a^8 - 7643a^6 + 2693a^4 + 7140a^2 + 864)x^3 \\
    &-9a^2(11a^{16} - 94a^{14} - 25a^{12} + 3068a^{10} - 13951a^8 + 25882a^6 - 15987a^4 - 9608a^2 + 14800)x^2 \\
    &+ 54a^2(a-2)^2(a-1)^2(a+1)^2(a+2)^2(a^2+1)(a^4 - a^3 - 5a^2 + 3a + 10)(a^4 + a^3 - 5a^2 - 3a + 10)x \\ 
    &-81a^2(a - 2)^4(a - 1)^4(a + 1)^4(a + 2)^4 \\
    &= 0.
\end{align*}
Then
\[
    N_{1/a}(d) \leq \frac{1}{2}(a^2-1)(a^2-2).
\]
Furthermore, when $a$ tends to infinity,
\[
    D_4(a) = 3a^2 + \frac{12}{\sqrt{5}}a - \frac{948}{25} + O(a^{-1}).
\]
\end{thm}
Theorem~\ref{thm: 4sdp_main} can be proved by using the alternative four-point semidefinite constraints with switching reduction, i.e., Theorem~\ref{thm: msdp_reduction} with $m=2$. In the following we fix $\a = 1/a$, $G = \begin{pmatrix}1 & \a \\ \a & 1\end{pmatrix}$, $\u_1 = \begin{pmatrix}\a & \a\end{pmatrix}^{\T}$ and $\u_2 = \begin{pmatrix}\a & -\a\end{pmatrix}^{\T}$. As shown in Example~\ref{eg: switching}, the two-point distance distribution of switching classes is $N_0[G] = N(N-1)$; the three-point distance distribution of switching classes is
\[
    y_1 \coloneqq N_1[1; \a, \a, \a], y_2 \coloneqq N_1[1; \a, \a, -\a];
\]
the four-point distance distribution of switching classes is
\[
    z_1 \coloneqq N_2[G; \u_1, \u_1, \a], z_2 \coloneqq N_2[G; \u_1, \u_1, -\a], z_3 \coloneqq N_2[G; \u_2, \u_2, -\a].
\]
By Theorem~\ref{thm: msdp_reduction},
\[
    \SDP^{d, 2}_0[G; X] = \begin{pmatrix}
        N(N-1) & y_1 & y_2 \\
        y_1 & y_1 + z_1 + z_2 & 2z_2 \\
        y_2 & 2z_2 & y_2 + z_2 + z_3
    \end{pmatrix} \succeq 0;
\]
for $k \geq 1$,
\[
    \SDP^{d, 2}_k[G; X] = \begin{pmatrix}
        \begin{matrix*}[l] \phantom{+} y_1 Q^{d, 2}_k(G; \u_1, \u_1, 1) \\+ z_1Q^{d, 2}_k(G; \u_1, \u_1, \a) \\+ z_2Q^{d, 2}_k(G; \u_1, \u_1, -\a) \end{matrix*} & \begin{matrix*}[l] \phantom{+} z_2Q^{d, 2}_k(G; \u_1, \u_2, \a) \\+ z_2Q^{d, 2}_k(G; \u_1, \u_2, -\a) \end{matrix*} \\[2em]
        \begin{matrix*}[l] \phantom{+} z_2Q^{d, 2}_k(G; \u_2, \u_1, \a) \\+ z_2Q^{d, 2}_k(G; \u_2, \u_1, -\a) \end{matrix*} & \begin{matrix*}[l] \phantom{+} y_2 Q^{d, 2}_k(G; \u_2, \u_2, 1) \\+ z_2Q^{d, 2}_k(G; \u_2, \u_2, \a) \\+ z_3Q^{d, 2}_k(G; \u_2, \u_2, -\a) \end{matrix*}
    \end{pmatrix} \succeq 0.
\]

Let $\langle \cdot, \cdot\rangle$ be the Frobenius inner product of matrices defined by $\langle A, B\rangle = \text{tr}(A^{\T}B)$. By the Schur product theorem, $\langle A, B\rangle \geq 0$ when both $A$ and $B$ are positive semidefinite. Our goal is to find the dual matrices $\{F_k\}$ such that
\[
    \sum_k \left\langle \SDP^{d, 2}_k[G; X], F_k \right\rangle \geq 0 \iff N \leq \frac{1}{2}(a^2-1)(a^2-2).
\]
If so, then $N_{1/a}(d)$ has an upper bound $(a^2-1)(a^2-2)/2$ whenever all the matrices $F_k$ are positive semidefinite.

\begin{lemma}\label{lemma: 4sdp_main}
There exists a symmetric $3 \times 3$ matrix $F$ and two real numbers $f_1$, $f_2$ such that
\[
    \left\langle \SDP^{d, 2}_0[G; X], F\right\rangle + \left\langle\SDP^{d, 2}_3[G; X], \begin{pmatrix} f_1 & 0 \\ 0 & f_2\end{pmatrix}\right\rangle = N(N-1)\left(\frac{1}{2}(a^2-1)(a^2-2)-2\right) - (y_1 + y_2).
\]
\end{lemma}
\begin{proof}
Let $F = \begin{pmatrix} F_0 & F_1 & F_2 \\ F_1 & F_3 & F_4 \\ F_2 & F_4 & F_5 \end{pmatrix}$. The matrix $F$ we found has null space
\[
    \R\begin{pmatrix}4 \\ (a+1)^3(a-2) \\ (a-1)^3(a+2)\end{pmatrix}.
\]
That is,
\begin{align}
    4F_1 + (a+1)^3(a-2)F_3 + (a-1)^3(a+2)F_4 &= 0 \tag{E1} \\
    4F_2 + (a+1)^3(a-2)F_4 + (a-1)^3(a+2)F_5 &= 0 \tag{E2}
\end{align}
On the other hand, to equate both sides of the equation, the coefficients of the variables in both sides should be the same. Therefore
\begin{align}
    N(N-1):\ & F_0 = \frac{1}{2}(a^2-1)(a^2-2)-2 \tag{E3} \\
    y_1:\ & (2F_1+F_3) + Q^{d,2}_3(G; \u_1, \u_1, 1)f_1 = -1 \tag{E4} \\
    y_2:\ & (2F_2+F_5) + Q^{d,2}_3(G; \u_2, \u_2, 1)f_2 = -1 \tag{E5} \\
    z_1:\ & F_3 + Q^{d,2}_3(G; \u_1, \u_1, \a)f_1 = 0 \tag{E6} \\
    z_2:\ & (F_3+4F_4+F_5) + Q^{d,2}_3(G; \u_1, \u_1, -\a)f_1 + Q^{d,2}_3(G; \u_2, \u_2, \a)f_2 = 0 \tag{E7} \\
    z_3:\ & F_5 + Q^{d,2}_3(G; \u_2, \u_2, -\a)f_2 = 0 \tag{E8}
\end{align}
(E1) to (E8) is a system of linear equations. When $(a^4-5a^2+12) - (a^2+7)d \neq 0$, we can solve $F_0, F_1, F_2, F_3, F_4, F_5, f_1, f_2$ in terms of $d$ and $a$. The values of $f_1$ and $f_2$ are
\begin{align*}
    f_1 &= \frac{a^3(d-3)\Big(3a(a-2)^2(a+1)^2 - (a^3+9a-6)d\Big)}{3d(a-2)(a-1)(a+1)\Big(a^4-5a^2+12 - (a^2+7)d\Big)}, \\
    f_2 &= \frac{-a^3(d-3)\Big(3a(a+2)^2(a-1)^2 - (a^3+9a+6)d\Big)}{3d(a+2)(a-1)(a+1)\Big(a^4-5a^2+12 - (a^2+7)d\Big)}.
\end{align*}
The principal minors of $F$ are
\begin{align*}
    F_0 &= \frac{1}{2}(a^2-1)(a^2-2)-2, \\
    F_3 &= \frac{(a-1)^3\Big(3(a+2)^2 - d\Big)}{a^3(a+1)^3(d-3)}f_1, \\
    F_5 &= \frac{-(a+1)^3\Big(3(a-2)^2 - d\Big)}{a^3(a-1)^3(d-3)}f_2, \\
    F_0F_3 - F_1^2 &= \frac{g_a(d)}{36d^2(a-2)^2(a+1)^6\Big(a^4-5a^2+12 - (a^2+7)d\Big)^2}, \\
    F_0F_5 - F_2^2 &= \frac{g_a(d)}{36d^2(a+2)^2(a-1)^6\Big(a^4-5a^2+12 - (a^2+7)d\Big)^2}, \\
    F_3F_5 - F_4^2 &= \frac{4g_a(d)}{9d^2(a-2)^2(a+2)^2(a-1)^6(a+1)^6\Big(a^4-5a^2+12 - (a^2+7)d\Big)^2}, \\
    \det F &= 0.
\end{align*}
Note that the polynomial $g_a(x)$ mentioned in the statement of Theorem~\ref{thm: 4sdp_main} appears as the numerator of the $2 \times 2$ principal minors of $F$.
\end{proof}

\begin{proof}[Proof of Theorem~\ref{thm: 4sdp_main}]
By Lemma~\ref{lemma: 4sdp_main}, if $f_1, f_2$ and all the principal minors of $F$ are nonnegative, then
\begin{align*}
    0 &\leq \left\langle \SDP^{d, 2}_0[G; X], F\right\rangle + \left\langle\SDP^{d, 2}_3[G; X], \begin{pmatrix} f_1 & 0 \\ 0 & f_2\end{pmatrix}\right\rangle = N(N-1)\left(\frac{1}{2}(a^2-1)(a^2-2) - N\right) \\
    &\Longrightarrow N \leq \frac{1}{2}(a^2-1)(a^2-2).
\end{align*}
For $a=3$, $D_4(3) \approx 14.42$ and $\lfloor D_4(3) \rfloor = 14$; for $a \geq 5$, one can check that the leading coefficient of $g_a(x)$ is negative, and the equation $g_a(x) = 0$ has four real roots in the disjoint intervals
\begin{align*}
    & \left[0, \frac{3}{2a^2}\right], \left[\frac{6}{7}a^2-8, \frac{6}{7}a^2\right], \\
    & \left[3a^2 - \frac{12}{\sqrt{5}}a - \frac{948}{25} + \frac{30\sqrt{5}}{a}, 3a^2 - \frac{12}{\sqrt{5}}a - \frac{948}{25} + \frac{45\sqrt{5}}{a}\right], \\
    & \left[3a^2 + \frac{12}{\sqrt{5}}a - \frac{948}{25} - \frac{32\sqrt{5}}{a}, 3a^2 + \frac{12}{\sqrt{5}}a - \frac{948}{25} + \frac{2\sqrt{5}}{a}\right],
\end{align*}
which proves the asymptotic argument of $D_4(a)$. In addition, we have
\[
    3a^2 + \frac{12}{\sqrt{5}}a - \frac{948}{25} - \frac{32\sqrt{5}}{a} - 1 \leq \lfloor D_4(a) \rfloor \leq 3a^2 + \frac{12}{\sqrt{5}}a - \frac{948}{25} + \frac{2\sqrt{5}}{a}.
\]
In either case, when taking $d = \lfloor D_4(a) \rfloor$, one can check that $f_1$, $f_2$ and all the principal minors of $F$ except $\det F$ are positive. Therefore
\[
    N_{1/a}(d) \leq \frac{1}{2}(a^2-1)(a^2-2)
\]
whenever $d \leq \lfloor D_4(a) \rfloor$.
\end{proof}

\begin{rmk}
Theorem~\ref{thm: 4sdp_main} explained the numerical results by~\citet{de2021k} partially. By the four-point semidefinite programming method, which is denoted by $\Delta_4(G)^{\ast}$ by~\cite{de2021k},
\[
    N_{1/5}(65) \leq 276,\ N_{1/7}(145) \leq 1128,\ N_{1/9}(251) \leq 3160,\ N_{1/11}(381) \leq 7140, 
\]
while the largest applicable dimensions we found are
\[
    \lfloor D_4(5) \rfloor = 64,\ \lfloor D_4(7) \rfloor = 144,\ \lfloor D_4(9) \rfloor = 250,\ \lfloor D_4(11) \rfloor = 380.
\]
\end{rmk}

\begin{rmk}
Lemma~\ref{lemma: 4sdp_main} can be understood as the duality of semidefinite programming. The found dual matrix $F$ paired with $\SDP^{d, 2}_0[G; X]$ is a matrix with the null space
\[
    \R\begin{pmatrix}1 \\ (a+1)^3(a-2)/4 \\ (a-1)^3(a+2)/4\end{pmatrix}.
\]
This is consistent with the fact that, all the known constructions attaining the bound in Theorem~\ref{thm: 4sdp_main} satisfy
\[
    \SDP^{d, 2}_0[G; X] = N(N-1)\begin{pmatrix}1 \\ (a+1)^3(a-2)/4 \\ (a-1)^3(a+2)/4\end{pmatrix}\begin{pmatrix}1 & \frac{(a+1)^3(a-2)}{4} & \frac{(a-1)^3(a+2)}{4}\end{pmatrix},
\]
which is a rank-$1$ matrix. On the other hand, the found dual matrix paired with $\SDP^{d, 2}_3[G; X]$ is a diagonal matrix. The off-diagonal entries are irrelevant, since the off-diagonal entry in $\SDP^{d, 2}_3[G; X]$ is
\[
    z_2\left(Q^{d, 2}_3(G; \u_1, \u_2, \a) + Q^{d, 2}_3(G; \u_1, \u_2, -\a)\right),
\]
which must be zero.
\end{rmk}

\subsection{Proof of Theorem~\ref{thm: maximum_mirror}}

\citet{glazyrin2018upper} proved that the constructions attaining the three-point semidefinite bound must lie in an $(a^2-2)$-dimensional subspace. As a generalization, Theorem~\ref{thm: maximum_mirror} states that the constructions attaining the four-point semidefinite bound must lie in an $(a^2-2)$-dimensional subspace as well. Therefore Theorem~\ref{thm: maximum_mirror} gives restrictions for the constructions of equiangular lines with cardinality $(a^2-1)(a^2-2)/2$ and inner product $1/a$ for dimensions $d \in (D_3(a), D_4(a)]$.

If a spherical projection $X$ is given by choosing an unit vector $b$ along one line, and then choosing unit vectors with positive inner products with $b$ along the other lines, then $X \setminus \{b\}$ is called the {\bf derived code} of the set of equiangular lines with respect to $b$.

The proof by~\citet{glazyrin2018upper} analyzes the structure of the derived codes, of which the two-point distance distributions are known from the linear programming method. However, our proof of Theorem~\ref{thm: 4sdp_main_mirror} cannot determine the distance distributions of derived codes directly. Instead, we consider the {\bf distance distribution with two fixed points}, and show that the distribution can be determined if the matrix $\SDP^{d, 2}_0[G; X]$ in the semidefinite constraint has rank $1$, which is the case when $d < D_4(a)$.

\begin{lemma}\label{lemma: 4sdp_rank}
Let $a \geq 3$ be an odd integer and $d < D_4(a)$. Suppose $X$ is a spherical projection of a set of equiangular lines in $\R^d$ with cardinality $(a^2-1)(a^2-2)/2$ and inner product $1/a$. Then the matrix $\SDP^{d, 2}_0[G; X]$ is of rank $1$.
\end{lemma}
\begin{proof}
Without the loss of generality, we may assume that $d = \lfloor D_4(a) \rfloor$. As in the proof of Theorem~\ref{thm: 4sdp_main},
\[
    0 \leq \left\langle \SDP^{d, 2}_0[G; X], F\right\rangle + \left\langle\SDP^{d, 2}_3[G; X], \begin{pmatrix} f_1 & 0 \\ 0 & f_2\end{pmatrix}\right\rangle = N(N-1)\left(\frac{1}{2}(a^2-1)(a^2-2)-N\right) = 0,
\]
therefore both $\left\langle \SDP^{d, 2}_0[G; X], F\right\rangle$ and $\left\langle\SDP^{d, 2}_3[G; X], \begin{pmatrix} f_1 & 0 \\ 0 & f_2\end{pmatrix}\right\rangle$ are zero. Recall that for two positive semidefinite $n \times n$ matrices $A, B$ with $\langle A, B\rangle = 0$, we have the inequality $\text{rank}(A) + \text{rank}(B) \leq n$, and all the eigenvectors of $A$ with positive eigenvalues are in $\ker B$. For $d = \lfloor D_4(a) \rfloor$, all the principal minors of $F$ except $\det F$ are positive, hence $F$ is a rank 2 matrix with null space
\[
    \R\begin{pmatrix}1 \\ (a+1)^3(a-2)/4 \\ (a-1)^3(a+2)/4\end{pmatrix},
\]
and the rank of the matrix $\SDP^{d, 2}_0[G; X]$ is at most $1$. Since the matrix is nonzero, the rank is $1$ and
\[
    \SDP^{d, 2}_0[G; X] = N(N-1)\begin{pmatrix}1 \\ (a+1)^3(a-2)/4 \\ (a-1)^3(a+2)/4\end{pmatrix}\begin{pmatrix}1 & (a+1)^3(a-2)/4 & (a-1)^3(a+2)/4\end{pmatrix}.
\]
\end{proof}

Let $b, b' \in X$ be different unit vectors. Consider the distance distribution with two fixed points
\begin{align*}
    N_{b, b'} &\coloneqq \#\{c \in X: G(b, b', c) \sim \begin{pmatrix}1 & \a & \a \\ \a & 1 & \a \\ \a & \a & 1\end{pmatrix}\}, \\
    N'_{b, b'} &\coloneqq \#\{c \in X: G(b, b', c) \sim \begin{pmatrix}1 & \a & \a \\ \a & 1 & -\a \\ \a & -\a & 1\end{pmatrix}\}.
\end{align*}
Clearly, we have $N'_{b, b'} = (N-2) - N_{b, b'}$. Furthermore, the two values can be determined when the matrix $\SDP^{d, 2}_0[G; X]$ is of rank $1$.
\begin{lemma}\label{lemma: 4sdp_fix_two_pt}
Suppose the matrix
\[
    \SDP^{d, 2}_0[G; X] = \begin{pmatrix}
        N(N-1) & y_1 & y_2 \\
        y_1 & y_1 + z_1 + z_2 & 2z_2 \\
        y_2 & 2z_2 & y_2 + z_2 + z_3
    \end{pmatrix}
\]
is of rank $1$. Then for any $b, b' \in X$ and $b \neq b'$,
\[
    N_{b, b'} = \frac{y_1}{N(N-1)} \text{ and } N'_{b, b'} = \frac{y_2}{N(N-1)}.
\]
\end{lemma}
\begin{proof}
Consider the matrix
\[
    \begin{pmatrix}
        1 & N_{b, b'} & N'_{b, b'} \\
        N_{b, b'} & (N_{b, b'})^2 & N_{b, b'}N'_{b, b'} \\
        N'_{b, b'} & N_{b, b'} N'_{b, b'} & (N'_{b, b'})^2
    \end{pmatrix}.
\]
It is not hard to see that the sum of such matrices over all $(b, b') \in X^2$ with $b \neq b'$ is $\SDP^{d, 2}_0[G; X]$. Since $\SDP^{d, 2}_0[G; X]$ is of rank $1$, we have
\[
    0 = N(N-1)(y_1+z_1+z_2) - y_1^2 = \left(\sum_{b \neq b'} 1 \right)\left(\sum_{b \neq b'} N_{b, b'}^2\right) - \left(\sum_{b \neq b'} N_{b, b'}\right)^2.
\]
Therefore, the values $N_{b, b'}$ for all $b \neq b'$ must be the same and equal $y_1 / N(N-1)$. Similarly, we have $N'_{b, b'} = y_2 / N(N-1)$.
\end{proof}

\begin{proof}[Proof of Theorem~\ref{thm: maximum_mirror}]
Suppose $a \geq 3$ is an odd integer, $d < D_4(a)$, and $X$ is a spherical projection of a set of equiangular lines in $\R^d$ with cardinality $(a^2-1)(a^2-2)/2$ and inner product $1/a$. Without the loss of generality, we assume that the spherical projection is chosen such that $X \setminus \{b\}$ is a derived code for some $b \in X$. Consider a graph $H$ with vertex set $V = X \setminus\{b\}$, and $(c, c') \in V^2$ forms an edge if and only if $c \cdot c' = \a$. We first show that the graph $H$ must be a strongly regular graph with parameters
\[
    \text{SRG}\left(\frac{a^2(a^2-3)}{2}, \frac{(a+1)^3(a-2)}{4}, \frac{(a+1)(a+2)(a^2-5)}{8}, \frac{(a+1)^3(a-2)}{8}\right).
\]

Indeed, the number of vertices is $v = |V| = N-1 = a^2(a^2-3)/2$. By Lemma~\ref{lemma: 4sdp_rank} and Lemma~\ref{lemma: 4sdp_fix_two_pt}, we have $N_{c, c'} = (a+1)^3(a-2)/4$ and $N'_{c, c'} = (a-1)^3(a+2)/4$ for any $c, c' \in X$ and $c \neq c'$. For any $b \in V$, the degree of $b$ in $H$ is
\[
    \deg_H(b) = \#\{c \in V: G(b_0, b, c) = \begin{pmatrix} 1 & \a & \a \\ \a & 1 & \a \\ \a & \a & 1\end{pmatrix}\} = N_{b_0, b} = \frac{(a+1)^3(a-2)}{4}.
\]
Therefore $H$ is $k$-regular with $k = (a+1)^3(a-2)/4$. For nonadjacent vertices $b, b' \in V$, let $\mu$ be the number of the common neighbors of $b$ and $b'$ in $H$. Then there are $(k-\mu)$ neighbors of $b$ which are not neighbors of $b'$, and $(k-\mu)$ neighbors of $b'$ which are not neighbors of $b$. Therefore
\[
    k = N_{b, b'} = (k-\mu) + (k-\mu) \Longrightarrow \mu = \frac{k}{2} = \frac{(a+1)^3(a-2)}{8}.
\]
By a similar argument, one can show that any two adjacent vertices has $\lambda = (3k-v-1)/2 = (a+1)(a+2)(a^2-5)/8$ common neighbors.

Let $M$ be the adjacency matrix of $H$. By the theory of strongly regular graphs, the eigenvalues of $M$ are $k$, $\frac{k}{a+1}$ and $-\frac{a+1}{2}$ with multiplicities $1$, $a^2-3$ and $v - a^2 + 2$, respectively. Since the derived code $X \setminus\{b_0\}$ is a spherical two-distance set in $\S^{d-2}$ with inner product set
\[
    \left\{\frac{\pm \a - \a^2}{1-\a^2}\right\} = \left\{\frac{\pm a - 1}{a^2 - 1}\right\} = \left\{\frac{1}{a+1}, \frac{-1}{a-1}\right\},
\]
the Gram matrix of the derived code is
\[
    I + \frac{1}{a+1}M - \frac{1}{a-1}(J-I-M).
\]
The eigenvalues of the Gram matrix are $\frac{a^2}{2}$ and $0$ with multiplicities $a^2-3$ and $v - a^2 + 3$, respectively. Therefore, the unit vectors in the derived code span an $(a^2-3)$-dimensional space, and the original set of equiangular lines lie in an $(a^2-2)$-dimensional space.

\end{proof}

The above proof indicates that the derived code of a construction attaining the four-point semidefinite bound must form a strongly regular graph with parameters
\[
    \text{SRG}\left(\frac{a^2(a^2-3)}{2}, \frac{(a+1)^3(a-2)}{4}, \frac{(a+1)(a+2)(a^2-5)}{8}, \frac{(a+1)^3(a-2)}{8}\right).
\]
Such strongly regular graphs are known to exist and are unique for $a=3$ and $a=5$ (see~\citet{seidel1968strongly, goethals1975regular}). Therefore, we can determine the uniqueness of some maximum constructions of equiangular lines.
\begin{cor}\label{cor: uniqueness}
The following maximum constructions of equiangular lines are unique up to orthogonal transformations.
\begin{enumerate}[label=(\alph*)]
    \item $28$ equiangular lines in $\R^d$ for $7 \leq d \leq 14$ with $\a = 1/3$.
    \item $276$ equiangular lines in $\R^d$ for $23 \leq d \leq 64$ with $\a = 1/5$.
\end{enumerate}
\end{cor}
By~\citet{glazyrin2018upper}, Theorem 4, the uniqueness is only known for $7 \leq d \leq 11$ with $\a = 1/3$, and $23 \leq d \leq 59$ with $\a = 1/5$.

\begin{rmk}
\citet{delsarte1977spherical} proved that $N_{1/3}(d) = 28$ for $7 \leq d \leq 15$. There are at least two maximum constructions in $\R^{15}$. Meanwhile,~\citet{cao2022lemmens} proved that $N_{1/5}(d) = 276$ for $23 \leq d \leq 185$. It is known that there are at least two maximum constructions in $\R^{185}$, so it remains unknown that whether the maximum constructions in $\R^d$ for $65 \leq d \leq 184$ are unique. 
\end{rmk}
\section{Further questions}
\label{sec: questions}

We have established the four-point semidefinite bound by using the alternative four-point semidefinite constraints with switching reduction, namely, Theorem~\ref{thm: msdp_reduction} with $m=2$. The five-point and the six-point semidefinite bounds may be able to be established in a similar way; however, the calculation may be quite complicated.
\begin{question}
Establish the five-point and the six-point semidefinite bounds for equiangular lines.
\end{question}

In the proof of the four-point semidefinite bound, we have also noticed some facts, but we fail to give the insights into them.
\begin{question}{\color{white}.}
\begin{enumerate}[label=(\alph*)]
    \item Give a geometric interpretation of the multi-point Gegenbauer polynomials $Q^{d, m}_k(G; \u, \v, t)$, and prove the switching property (Proposition~\ref{prop: Q_switching}) with this interpretation.
    \item Give an explanation for the fact that the four-point semidefinite bound for equiangular lines only depends on the constraints $\SDP^{d, 2}_k[G; X]$ for $k=0$ and $k=3$, but not any other $k$.
    \item Give a characterization of the set of equiangular lines such that the matrix $\SDP^{d, m}_k[G; X]$ vanishes for some certain $m$, $k$ and proper Gram matrix $G$.
\end{enumerate}
\end{question}

For odd integers $a \geq 3$, the existences of constructions of $(a^2-1)(a^2-2)/2$ equiangular lines in $\R^{a^2-2}$ with inner product $\a = 1/a$ are of much importance. By Theorem~\ref{thm: 4sdp_main_mirror}, if such a construction is found for some $a$, then the values of $N_{1/a}(d)$ are answered for a wide range of $d$. However, such constructions are only known for $a=3$ and $a=5$.
\begin{question}\label{q: tight_design}
For which odd integers $a \geq 3$ does the maximum cardinality $N_{1/a}(a^2-2)$ equal $(a^2-1)(a^2-2)/2$? If the cardinality $(a^2-1)(a^2-2)/2$ is not attainable, what is the value of $N_{1/a}(a^2-2)$?
\end{question}
The specific family of equiangular lines is also related to tight spherical $5$-designs and tight spherical $4$-designs introduced by~\citet{delsarte1977spherical}. They proved that a tight spherical $5$-design in $\S^{d-1}$ is a both-end spherical projection of a set of equiangular lines in $\R^d$ with cardinality $d(d+1)/2$ and inner product $\a = 1/\sqrt{d+2}$, and a tight spherical $4$-design in $\S^{d-2}$ is a derived code of a set of equiangular lines in $\R^d$ with the same parameters. By Neumann's Theorem, such constructions of equiangular lines exist only if $d \leq 3$ or $1/\a$ is an odd integer.

Question~\ref{q: tight_design} is equivalent to classify tight spherical $5$-designs and tight spherical $4$-designs. The only known constructions for tight spherical $5$-designs are in $\S^2$, $\S^6$ and $\S^{22}$; the only known constructions for tight spherical $4$-designs are in $\S^1$, $\S^5$ and $\S^{21}$. Furthermore, \citet{bannai2005nonexistence, nebe2013tight} proved that the tight spherical $5$-designs in $\S^{a^2-3}$ as well as tight spherical $4$-designs in $\S^{a^2-4}$ do not exist for $a = 7, 9, 13, 21, 25, 45, 57, 61, 69, 85, 93, \dots$. Our numerical experiments suggest that the multi-point semidefinite programming method for $m \leq 4$ may not be able to show the nonexistence for any other $a$.

On the other hand, for $a=5$ \citet{cao2022lemmens} proved that $N_{1/5}(d) = 276$ for $23 \leq d \leq 185$. Corollary~\ref{cor: uniqueness} indicated that the constructions of $276$ equiangular lines in $\R^d$ for $23 \leq d \leq 64$ are unique. However, the number of different maximum constructions for $d > 64$ is not known well.
\begin{question}
What is the number of different constructions of $276$ equiangular lines in $\R^d$ for $65 \leq d \leq 185$ with $\a = 1/5$?
\end{question}

As indicated in Remark~\ref{rmk: alternative_simpler}, the alternative semidefinite constraints $\SDP^d_k(X)$ is simpler than the original one developed by~\citet{bachoc2008new} when the concerned object $X$ is a spherical $s$-distance set. The alternative constraints may be also useful when concerning objects other than sets of equiangular lines.
\begin{question}
Derive new results of the maximum cardinality of spherical $s$-distance sets or new non-existences of strongly regular graphs using the alternative semidefinite constraints.
\end{question}

\section*{Acknowledgement}
We thank Henry Cohn for initializing the discussion of the four-point semidefinite constraints. We thank Hung-Hsun Hans Yu for his invaluable suggestions in writing. W.-H. Yu was supported by MOST in Taiwan under Grant MOST109-2628-M-008-002-MY4. W.-J. Kao is supported by National Center for Theoretical Sciences (No. 111L104040).

\appendix
\section{Research review on equiangular lines}
\subsection{Previous results on $N(d)$}
\label{sec: review}
Table~\ref{tb: Nd_detail} records the current known lower and upper bounds for $N(d)$ as well as their references.

\begin{table}[h]
    \centering
    \begin{tabular}{c|c|c|c}
        $d$ & $N(d)$ & construction & upper bound \\ \hline
        2 & 3 & & \\
        3-4 & 6 & \citet{haantjes1948equilateral} & \citet{haantjes1948equilateral} \\
        5 & 10 & \citet{van1966equilateral} & \citet{van1966equilateral} \\
        6 & 16 & \citet{van1966equilateral} & \citet{van1966equilateral} \\
        7-13 & 28 & \citet{van1966equilateral} & \citet{lemmens1973equiangular} \\
        14 & 28 & \citet{van1966equilateral} & \citet{greaves2021equiangular} \\
        15 & 36 & \citet{bussemaker1991symmetric} & \citet{lemmens1973equiangular} \\
        16 & 40 & \citet{higman1964finite} & \citet{greaves2021equiangular} \\
        17 & 48 & \citet{lemmens1973equiangular, greaves2016equiangular} & \citet{greaves2021equiangular_1718} \\
        18 & 57-60 & \citet{greaves2021equiangular_1718} & \citet{greaves2018equiangular} \\
        19 & 72-74 & \citet{taylor1971some} & \citet{greaves2021equiangular} \\
        20 & 90-94 & \citet{taylor1971some} & \citet{greaves2021equiangular} \\
        21 & 126 & \citet{taylor1971some} & \citet{lemmens1973equiangular} \\
        22 & 176 & \citet{goethals1970strongly, taylor1971some} & \citet{lemmens1973equiangular} \\
        23 & 276 & \citet{goethals1970strongly, taylor1971some} & \citet{lemmens1973equiangular} \\
        24-41 & 276 & \citet{goethals1970strongly, taylor1971some} & \citet{barg2013new_eq, yu2017new} \\
        42 & 276-288 & \citet{goethals1970strongly, taylor1971some} & \citet{barg2013new_eq, yu2017new} \\
        43 & 344 & \citet{taylor1971some} & \citet{barg2013new_eq, yu2017new}
    \end{tabular}
    \caption{Lower and upper bounds for $N(d)$}
    \label{tb: Nd_detail}
\end{table}

Below is a brief review for the bounds for $N(d)$ in small dimensions:
\begin{itemize}
\item \citet{haantjes1948equilateral} proved that $N(3) = N(4) = 6$.
\item \citet{van1966equilateral} determined $N(5)$ and $N(6)$ and constructed $28$ equiangular lines in $\R^7$. They also proved that $N_{1/a}(d) \leq d(a^2-1)/(a^2-d)$ for $d < a^2$, which is known as the {\bf relative bound}.
\item Gerzon~\cite{lemmens1973equiangular} proved that $N(d) \leq d(d+1)/2$. This bound is known as the {\bf absolute bound}. The equality occurs only if $d = 1, 2, 3$, or $a^2-2$ for some odd integer $a$. Neumann~\cite{lemmens1973equiangular} proved that a set $X$ of equiangular lines in $\R^d$ with $|X| > 2d$ must have inner product $\a$ being reciprocal of an odd integer.
\item \citet{lemmens1973equiangular} collected the maximum known constructions of equiangular lines in dimensions $15$-$17$, $19$-$23$ and $43$. Most of the constructions can be found in~\citeauthor{taylor1971some}'s thesis~\cite{taylor1971some}. An explicit construction of $48$ equiangular lines in $\R^{17}$ can be found in~\citet{greaves2016equiangular}. \citeauthor{lemmens1973equiangular} also determined $N(d)$ for $d = 7$-$13$, $15$ and $21$-$23$ by Neumann's theorem, the relative bound and some analysis on pillars.
\item \citet{barg2013new_eq} showed numerically that $N_{1/5}(d) = 276$ for $23 \leq d \leq 60$ by the {\bf semidefinite programming method}. The results can determine that $N(d) = 276$ for $24 \leq d \leq 41$, $N(42) \leq 288$ and $N(43) = 344$. A rigorous proof of the numerical results is given by~\citet{yu2017new}. The theorem is Theorem~\ref{thm: 3sdp_main_mirror} in the present article.
\item The maximum known construction in $\R^{18}$ was the same as $\R^{17}$ until~\citet{szollHosi2019remark} found a construction of equiangular lines with cardinality $54$. Later,~\citet{lin2020saturated} and~\citet{greaves2021equiangular_1718} found constructions with cardinalities $56$ and $57$, respectively.
\item For $d = 14, 16$-$20$, all the best possible constructions have inner products $\a = 1/5$ (\citet{lemmens1973equiangular}). The following are the improvements of the upper bounds:
\begin{enumerate}[label=\arabic*.]
\item $N(14)$: $30$ (relative bound) $\longrightarrow 29$ (\citet{greaves2016equiangular}) $\longrightarrow 28$ (\citet{greaves2021equiangular}).
\item $N(16)$: $42$ (relative bound) $\longrightarrow 41$ (\citet{greaves2016equiangular}) $\longrightarrow 40$ (\citet{greaves2021equiangular}).
\item $N(17)$: $51$ (relative bound) $\longrightarrow 50$ (\citet{sustik2007existence}) $\longrightarrow 49$ (\citet{greaves2019equiangular}) $\longrightarrow 48$ (\citet{greaves2021equiangular_1718}).
\item $N(18)$: $61$ (relative bound) $\longrightarrow 60$ (\citet{greaves2018equiangular}).
\item $N(19)$: $76$ (relative bound) $\longrightarrow 75$ (\citet{azarija2018there}) $\longrightarrow 74$ (\citet{greaves2021equiangular}).
\item $N(20)$: $96$ (relative bound) $\longrightarrow 95$ (\citet{azarija2020there}) $\longrightarrow 94$ (\citet{greaves2021equiangular}).
\end{enumerate}
\end{itemize}

For the asymptotic lower bound, \citet{taylor1971some} proved that $N(q^2-q+1) \geq q^3+1$ for odd prime powers $q$. This construction indicates that $N(d) = \Omega(d^{3/2})$. \citet{de2000large} proved that $N(d) \geq \frac{2}{9}(d+1)^2$ for $d = 6 \cdot 4^i - 1$, which implies that
\[
    \frac{2}{9} \leq \mathop{\lim\sup}_{d \to \infty}\frac{N(d)}{d^2} \leq \frac{1}{2}.
\]
\citet{greaves2016equiangular} offered another construction, indicating that
\[
    N(d) \geq \frac{32d^2+328d+296}{1089}.
\]

\subsection{Previous results on $N_{\a}(d)$}
\label{sec: review2}
Neumann's theorem starts the study on $N_{1/a}(d)$ for odd integers $a$. Below is a brief review of the researches on $N_{\a}(d)$.
\begin{itemize}

\item For $\a = 1/3$, \citet{lemmens1973equiangular} proved that $N_{1/3}(d) = 28$ for $7 \leq d \leq 15$, and $N_{1/3}(d) = 2(d-1)$ for $d \geq 15$.
\item For $\a = 1/5$, \citet{lemmens1973equiangular} conjectured that $N_{1/5}(d)$ equals $276$ for $23 \leq d \leq 185$ and $\textstyle\lfloor \frac{3}{2}(d-1) \rfloor$ for $d \geq 185$.
\begin{itemize}
    \item \citet{neumaier1989graph} proved the conjecture for sufficiently large $d$.
    \item \citet{yu2017new} proved the conjecture for $23 \leq d \leq 60$.
    \item \citet{lin2020equiangular} proved that the conjecture is true when the base size $K = 2, 3, 5$, as \citeauthor{lemmens1973equiangular} claimed.
    \item Recently, \citet{cao2022lemmens} completely proved the conjecture.
\end{itemize}

\item For the asymptotic behavior of $N_{\a}(d)$ as $d \to \infty$,
\begin{itemize}
    \item \citet{bukh2016bounds} proved that $N_{\a}(d) = O(d)$ for any $\a \in (0, 1)$.
    \item \citet{balla2018equiangular} proved that $N_{\a}(d) \leq 1.93d$ if $\a \neq 1/3$.
    \item \citet{jiang2020forbidden} proved that $N_{1/5}(d) = (3/2)d + O(1)$, $N_{1/(1+2\sqrt{2})}(d) = (3/2)d + O(1)$, and $N_{\a}(d) \leq 1.49d$ for every $\a \neq 1/3, 1/5$ and $1/(1+2\sqrt{2})$ and for sufficiently large $d$.
    \item \citet{jiang2021equiangular} proved that $N_{\a}(d) = \lfloor\frac{k}{k-1}(d-1)\rfloor$ for all sufficiently large $d$, where $k$ is the spectral radius order of $(1-\a)/(2\a)$. In particular, $N_{1/a}(d) = \lfloor\frac{a+1}{a-1}(d-1)\rfloor$ for all odd integers $a \geq 3$ and all sufficiently large $d$. If $k = \infty$, then $N_{\a}(d) = d + o(d)$.
\end{itemize}

\item Some new relative bounds.
\begin{itemize}
    \item \citet{okuda2016new} proved that
    \[
        N_{1/a}(d) \leq 2 + (d-2) \max\left(\frac{(a-1)^3}{d-(3a^2-6a+2)}, \frac{(a+1)^3}{(3a^2+6a+2)-d}\right)
    \]
    for $3a^2 - 6a + 2 < d < 3a^2 + 6a + 2$.
    \item \citet{glazyrin2018upper} proved that
    \[
        N_{1/a}(d) \leq \left(\frac{2}{3}a^2 + \frac{4}{7}\right)d + 2
    \]
    for $a \geq 3$.
    \item \citet{balla2021equiangular} proved that
    \[
        N_{1/a}(d) \leq \frac{a^3}{2}\sqrt{d} + \frac{a+1}{2}d
    \]
    and
    \[
        N_{1/a}(d) \leq \max\left(2a^5+\frac{2a^4}{a-1}, (2+\frac{8}{(a-1)^2})(d+1)\right).
    \]
\end{itemize}

\end{itemize}

\bibliographystyle{plainnat}
\bibliography{ref}

\end{document}